\documentclass[a4paper, 12pt]{amsart}
\usepackage{amssymb, amsmath, amsthm, enumerate, color, subcaption, tipa, upgreek, xcolor}
\usepackage{amstext, tikz, appendix, mathrsfs, mathabx, mathtools, pgfplots, stmaryrd}
\usepackage{csquotes}
\usepackage[colorlinks=true, breaklinks=true, linkcolor=black, citecolor=blue, urlcolor=red]{hyperref} 
\usepackage[english]{babel}
\usepackage{forest, adjustbox}
\usepackage{tikz-cd}

\numberwithin{equation}{section}
\newtheorem{letterthm}{Theorem}

\newtheorem{lettercor}[letterthm]{Corollary}

\newtheorem{theorem}{Theorem}[section]

\newtheorem{lemma}[theorem]{Lemma}
\newtheorem{corollary}[theorem]{Corollary}
\newtheorem{proposition}[theorem]{Proposition}

\newtheorem{observation}[theorem]{Observation}

\theoremstyle{definition} 
\newtheorem{definition}[theorem]{Definition}
\newtheorem{notation}[theorem]{Notation}
\newtheorem{remark}[theorem]{Remark}
\newtheorem{example}[theorem]{Example}

\newcommand{\act}{\curvearrowright}

\newcommand{\cA}{\mathcal A}

\DeclareMathOperator{\atom}{atom}

\newcommand{\cC}{\mathcal C}

\newcommand{\C}{\mathbf C}

\DeclareMathOperator{\Comm}{Comm}
\DeclareMathOperator{\comp}{comp}

\DeclareMathOperator{\diff}{diff}

\newcommand{\cF}{\mathcal F}

\newcommand{\fH}{\mathfrak H}
\newcommand{\fHatom}{\fH_{\atom}}
\newcommand{\fHray}[1]{\fH_{\atom}^{[#1]}}

\newcommand{\scrH}{\mathscr H}
\DeclareMathOperator{\Hilb}{Hilb}

\DeclareMathOperator{\id}{id}
\DeclareMathOperator{\Ind}{Ind}

\DeclareMathOperator{\Irr}{Irr}

\newcommand{\into}{\hookrightarrow}

\newcommand{\fK}{\mathfrak K}
\newcommand{\scrK}{\mathscr K}

\DeclareMathOperator{\Leaf}{Leaf}
\DeclareMathOperator{\length}{length}

\DeclareMathOperator{\Mod}{Mod}

\newcommand{\N}{\mathbf{N}}

\newcommand{\NInd}{\text{Ind-mixing}}

\newcommand{\cO}{\mathcal O}

\newcommand{\ot}{\otimes}

\newcommand{\cP}{\mathcal P}

\DeclareMathOperator{\PSU}{PSU}

\newcommand{\R}{\mathbf{R}}

\DeclareMathOperator{\Rep}{Rep}
\DeclareMathOperator{\res}{res}
\DeclareMathOperator{\Root}{Root}

\newcommand{\cspan}{\overline{\textrm{span}}}
\DeclareMathOperator{\Span}{span}

\newcommand{\sub}[1]{\langle#1\rangle}

\newcommand{\cT}{\mathcal T}
\newcommand{\fT}{\mathfrak T}

\newcommand{\ti}{\tilde}

\newcommand{\cU}{\mathcal U}

\newcommand{\varep}{\varepsilon}
\DeclareMathOperator{\Ver}{Ver}

\newcommand{\fX}{\mathfrak X}
\newcommand{\scrX}{\mathscr X}

\newcommand{\Y}{\wedge}
\newcommand{\Z}{\mathbf{Z}}
\newcommand{\fZ}{\mathfrak Z}
\newcommand{\onto}{\twoheadrightarrow}
\DeclarePairedDelimiterX{\norm}[1]{\lVert}{\rVert}{#1}

\providecommand{\keywords}[1]{\textbf{\textit{Index terms---}} #1}

\begin{document}
	
\title[Atomic Pythagorean representations]
{Atomic representations of R.~Thompson's groups and Cuntz's algebra}
\thanks{
AB is supported by the Australian Research Council Grant DP200100067.\\
DW is supported by an Australian Government Research Training Program (RTP) Scholarship.}
\author{Arnaud Brothier and Dilshan Wijesena}
\address{Arnaud Brothier, Dilshan Wijesena\\ School of Mathematics and Statistics, University of New South Wales, Sydney NSW 2052, Australia}
\email{arnaud.brothier@gmail.com\endgraf
	\url{https://sites.google.com/site/arnaudbrothier/}}

\begin{abstract}
We continue to study Pythagorean unitary representation of Richard Thompson's groups $F,T,V$ and their extension to the Cuntz(--Dixmier) algebra $\cO$. Any linear isometry from a Hilbert space to its direct sum square produces such. We focus on those arising from a finite-dimensional Hilbert space.
We show that they decompose as a direct sum of a so-called diffuse part and an atomic part.
We previously proved that the diffuse part is Ind-mixing: it does not contain induced representations of finite-dimensional ones.
In this article, we fully describe the atomic part: it is a finite direct sum of irreducible monomial representations arising from a precise family of parabolic subgroups.
\end{abstract}

\maketitle


%%%%%%%%%%%%%%%%INTRODUCTION%%%%%%%%%%%%%%%%%%%%%%%%%%%%%%%%%
%%%%%%%%%%%%%%%%%%%%%%%%%%%%%%%%%%%%%%%%%%%%%%%%%%%%%%%%%%

\section*{Introduction}
Richard Thompson's groups $F\subset T\subset V$ are fascinating groups which appear in various branches of mathematics, see \cite{Cannon-Floyd-Parry96}.
Groups are often understood via their actions.
Jones' technology offers a practical machinery to construct such by leveraging that $F,T,V$ are fraction groups of basic categories \cite{Jones17,Jones18a}. 
This approach has already been successfully applied for constructing actions on operator algebras and groups, and unitary representations \cite{Jones18b, Brothier-Stottmeister20,Brothier22,Brothier21,Jones21,Brothier-Jones19b}. 
Beyond producing actions of the Thompson groups we may use this technology to produce new knot invariants, obtain natural subgroups of the Thompson groups, and to study certain non-commutative probabilities \cite{Jones18a,Grymski-Peters22,Golan-Sapir17,Aiello-Nagnibeda23,Kostler-Krishnan-Wills20,Kostler-Krishnan22}.
Finally, this technology is useful for studying other Thompson-like groups built from categories \cite{Brothier21,Brothier23a,Brothier23b,Larsen23}.
We refer the reader to the recent surveys \cite{Jones19, Brothier20, Aiello22}.

%Pyth representations
The first author and Jones considered the following particular case of Jones' technology \cite{Brothier-Jones19a}: any linear isometry $R:\fH\to\fH\oplus\fH$, with $\fH$ a complex Hilbert space, permits to construct a unitary  representation $(\sigma^V,\scrH)$ of the largest Thompson group $V$. We write $\sigma^F,\sigma^T$ for the restrictions to the subgroup $F,T$, respectively. We call these {\it Pythagorean representations} (in short P-representations).
The isometric condition translates into
\begin{equation}\label{eq:Pyth}
	A^*A+B^*B=\id_\fH, \tag{PE}
\end{equation}
where $A,B\in B(\fH)$ are the legs of $R$.
We call $(A,B,\fH)$ a {\it Pythagorean module} (in short P-module). 
P-modules correspond to the representations of the {\it Pythagorean algebra} $\cP$: the universal $C^*$-algebra defined by Relation \ref{eq:Pyth}.
We consider maps between P-modules that only intertwinnes the $A$'s and $B$'s (but not necessarily their adjoints as one would classically require in $\Rep(\cP)$). 
This develops a highly non-trivial representation theory that allows us to perform powerful classifications of P-representations by solely working with P-modules. 
Beyond classification we are moreover able to read properties of P-representations by only studying the operators $A,B$.

By adding the relations $AA^*=BB^*=\id_\fH$ we obtain a quotient $\cP\onto\cO$ on the Cuntz(--Dixmier) algebra \cite{Dixmier64,Cuntz77}.
Surprisingly, any representation of $\cP$ canonically lifts into a representation $\sigma^\cO$ of $\cO$.
Moreover, $\sigma^\cO$ restricts into $\sigma^V$ (after identifying $V$ inside $\cO$ via the Birget--Nekrashevych embedding, see Section \ref{sec:cuntz-algebra} and \cite{Birget04,Nekrashevych04}).

%Finite-dimensional representation
In practice we mostly consider $\fH$ finite-dimensional, yet the space $\scrH$ on which $V$ and $\cO$ act is always infinite-dimensional and is roughly equal to trees with leaves decorated by vectors in $\fH$.
Thus, the force of this construction resides in constructing and studying representations of $F,T,V$ and $\cO$ using only finite-dimensional data. 

{\bf \underline{Brief outline of the article.}}
In this article we continue our systematic study of P-representations initiated in \cite{Brothier-Wijesena22}.
For improving the clarity of the exposition we restrict our study to {\it finite-dimensional} P-modules in this article. Several of our results extend to the infinite-dimensional case and will be proven in a future article.
We previously defined {\it diffuse} P-modules (i.e.~increasing words in $A,B$ tend to zero for the strong operator topology) and proved that the associated P-representation (also named {\it diffuse}) $\sigma^X$ of the Thompson groups $X=F,T,V$ are {\it Ind-mixing} (i.e.~$\Ind_H^X\theta\not\subset \sigma^X$ for all subgroups $H\subset X$ and finite-dimensional representations $\theta:H\to \cU(\C^d)$) \cite{Brothier-Wijesena22}.
In this present article we define a negation of being diffuse called {\it atomic}. 
We show that any P-representation $\sigma$ decomposes into $\sigma_{\atom}\oplus \sigma_{\diff}$ where $\sigma_{\atom},\sigma_{\diff}$ are themselves P-representations named the atomic and diffuse parts of $\sigma$, respectively. 
We then decompose $\sigma_{\atom}$ into explicit irreducible components. This is achieved by solely decomposing the underlying P-module.

{\bf \underline{Detailed content of the article and main results.}}
\begin{center}{\bf For the rest of the introduction all P-modules are finite-dimensional and all P-representations are built from finite-dimensional P-modules.}\end{center}

Decomposing and classifying P-modules is more much subtle than the usual representation theory of the $C^*$-algebra $\cP$.
However, this is a small cost to pay as it allows one to decompose P-representations of $F,T,V,\cO$ at the level of $\cP$.
A P-module $m=(A,B,\fH)$ does not decompose in general as a direct sum of irreducible components.
However, we have the (orthogonal) direct sum 
\[\fH = \fH_{\comp} \oplus \fH_{\res}\]
where {\it the residual subspace} $\fH_{\res}$ is the largest vector subspace that does not contain any non-trivial sub-module and $\fH_{\comp}$ is {\it the complete sub-module}. 
We will see that $\fH_{\comp}$ can be decomposed into irreducible components.

\textbf{Diffuse and atomic P-modules.}
If $p$ is an infinite binary sequence (often called a {\it ray} when identified with a path in the infinite rooted binary tree), then write $[p]$ for its class obtained by swapping finite prefix.
We set $p_n$ to be the first $n$ digits of $p$ that we often identify with an operator obtained by replacing digits of $p$ by $A,B$ and reversing the order.
We define $\fH_{\diff}\subset\fH_{\comp}$ to be the subset of vectors $\xi$ satisfying that $\lim_n \norm{p_n \xi}=0$ for all rays $p$.
This forms a sub-module that we call the \textit{diffuse} part of $\fH$.
In contrast, for each periodic ray $p$ we define $\fHray{p} \subset \fH_{\comp}$ to be the span of all vectors $\xi$ such that there exists a ray $q \in [p]$ satisfying $\norm{q_n\xi} = \norm{\xi}$ for all $n \geq 1$. 
This also forms a sub-module $\fH_{\atom} := \oplus_{[p]} \fHray{p}$ called the \textit{atomic} part of $\fH$. 

{\bf The Pythagorean functor.}
Jones' technology promotes a P-module into a representation of $F$ that extends to $T,V$ and even $\cO$.
This process is functorial giving the four {\it Pythagorean functors} (P-functors) 
$\Pi^X: \Mod(\cP) \rightarrow \Rep(X)$ for $X=F,T,V,\cO$, where $\Rep(X)$ is the usual category of representations of $X$, and $\Mod(\cP)$ is the category of P-modules with same objects as $\Rep(\cP)$ but with more morphisms, see Section \ref{subsec:p-functor}.
We can now state our first main theorem on decomposing P-modules and P-representations.

\begin{letterthm} \label{theo:atom-diff-decomp}
If $(A,B,\fH)$ is a P-module then the following assertions hold:
	\begin{enumerate}[i]
		\item $\Pi^X(\fH) \cong \Pi^X(\fH_{\comp})$;
		\item $\fH_{\comp} = \fH_{\atom} \oplus \fH_{\diff} = \oplus_{[p]} \fHray{p} \oplus \fH_{\diff}$;
		\item $\Pi^X(\fH) \cong \Pi^X(\fH_{\atom}) \oplus \Pi^X(\fH_{\diff}) \cong \oplus_{[p]} \Pi^X(\fHray{p}) \oplus \Pi^X(\fH_{\diff})$
	\end{enumerate}
	where $[p]$ runs over all {\it periodic rays} and $X=F,T,V,\cO$.	
\end{letterthm}
The first statement is proven in Proposition \ref{prop:complete-mod-rep} and motivates the terminology. 
The second statement is proven in Theorem \ref{thm:atomic-diff-decomp} and Proposition \ref{prop:atom-decomp-rays}. The third statement follows from the first two.
Note that we only consider (eventually) periodic rays. In fact, we may even restrict to classes of rays with period of length smaller than $\dim(\fH)$.

\textbf{Description of atomic P-representations.}
The second half of the paper is dedicated to precisely describing and classifying atomic representations: P-representations from atomic P-modules. 
Given $d\geq 1$ define $W_d$ to be a set of representatives of prime binary word of length $d$ modulo cyclic permutations and write $S_1$ for the circle (complex numbers of modulus $1$).
For each pair $(w,\varphi)\in W_d\times S_1$ we define an explicit P-module $m_{w,\varphi}$ using $d$ by $d$ matrices.
This P-module is atomic and irreducible. 
We show that conversely all irreducible atomic P-module is of this form (up to isomorphism) and moreover describe explicitly their associated P-representations $\Pi^X(m_{w,\varphi})$ for the Thompson groups $X=F,T,V$. 
These explicit descriptions together with the Mackey--Schoda criterion give us a complete comprehension of atomic representations.
We refer to Section \ref{subsec:atomic-pmod-model} for notations and details.
In the below theorem, statements i, ii, v are proven in Theorem \ref{thm:atom-rep-class} while statements iii, iv follow easily from Section \ref{subsec:atomic-rep-fam}.

\begin{letterthm}\label{theo:atom-rep-classif}
	Let $X=F,T,V$ and fix $w \in W_d$, $\varphi \in S_1$. 
Write $p$ for the periodic ray $w^\infty$,  $X_p$ for the parabolic subgroup $\{g\in X:\ g(p)=p\}$, $\chi_{\varphi}^p$ for the representation $X_p\to S_1, g\mapsto \varphi^{\log(2^{|p|})(g'(p))}$.
Then the following assertions are true.
	\begin{enumerate}[i]
		\item If $(X, \vert w \vert) \neq (F,1)$ , then
		$\Pi^X(m_{w, \varphi})\cong \Ind_{X_p}^X\chi_\varphi^p$ and this representation is irreducible.
		\item If $w=0$ (resp.~$w=1$) set $q=1\cdot 0^\infty$ (resp.~$0\cdot 1^\infty$). We have:
		\[\Pi^F(m_{w, \varphi})\cong \chi_\varphi^p \oplus \Ind_{F_{q}}^q\chi_\varphi^q\]
		which is a direct sum of a one-dimensional representation and an irreducible one.
		\item Given $(w,\varphi)$ and $(v,\mu)$ and assuming that $(X,|w|)\neq (F,1)$, then $\Pi^X(m_{w, \varphi})\cong \Pi^X(m_{v, \mu})$ when $(w,\varphi)=(v,\mu).$
		\item Given $(w,\varphi)$ and $(v,\mu)$ and assuming that $|w|=|v|=1$, then $\Pi^F(m_{w, \varphi})\cong \Pi^F(m_{v, \mu})$ when $(w,\varphi)=(v,\mu)$ or $\varphi=\mu=1$.
		\item Every atomic representation is a finite direct sum of irreducible ones appearing in:
		\begin{equation}\label{eq:list-FD}\tag{AR}\{ \Ind_{X_p}^X\chi_\varphi^p:\ p \text{ eventually periodic ray}, \varphi\in S_1 \}.\end{equation}
	\end{enumerate}
\end{letterthm}

This classification permits
to deduce that, up to few exceptions, the P-functors preserve irreducible classes in the atomic case (see Section \ref{subsec:atom-rep-decomp}). 
Moreover, this extends to the Cuntz algebra (since $V\subset\cO$ and using items i, iii).
We will obtain a similar conclusion in the {\it diffuse} case in \cite{Brothier-Wijesena23} by using a direct conceptual argument.

\begin{lettercor} \label{cor:intro-classif-atom-general}
	Let $X = F,T,V,\cO$ and $\fH_1, \fH_2$ be two atomic P-modules. If either $X \neq F$ or both $\fH_1, \fH_2$ do not contain a copy of $m_{w, \varphi}$ with $\vert w \vert = 1$ then:
	\begin{enumerate}[i]
		\item $\Pi^X(\fH_1)$ is irreducible if and only if $(\fH_1)_{\comp}$ is irreducible;
		\item $\Pi^X(\fH_1) \cong \Pi^X(\fH_2)$ if and only if $(\fH_1)_{\comp}\cong (\fH_2)_{\comp}$.
	\end{enumerate}
\end{lettercor}

Recall that a representation of a group is {\it weakly mixing} (resp.~{\it Ind-mixing}) if it does not contain (resp.~the induction of) a non-zero finite-dimensional representation.
We proved in \cite{Brothier-Wijesena22} that a {\it diffuse} P-representation is Ind-mixing.
In sharp contract, atomic representations are direct sum of monomial representations. 
This allows to deduce the following characterisations (proved in Section \ref{subsec:atom-rep-decomp}).

\begin{lettercor}\label{cor:intro-ind-weak-mix}
Fix a P-module	$m = (A,B,\fH)$ and set $X=F,T,V$. We have:
	\begin{enumerate}[i]
		\item the representation $\Pi^X(m)$ is weakly mixing if and only if either $X=T,V$ or $\lim_{n}A^n\xi=\lim_n B^n\xi=0$ for all $\xi\in\fH$;
		\item the representation $\Pi^X(m)$ is Ind-mixing if and only if $m$ is diffuse.
	\end{enumerate}
\end{lettercor}

{\bf Geometrical interpretation.}
We now introduce coordinates: $\fH=\C^d$ and $A,B\in M_d(\C)$ (they are $d$ by $d$ complex matrices).
Denote by $\Irr_{\atom}(d)$ the set of P-modules $m=(A,B,\C^d)$ that are atomic and irreducible.
Our work demonstrates that in the atomic case the number $d$ is an invariant of both $m$ and $\Pi^X(m)$ that we name the {\it Pythagorean dimension} (in short P-dimension).
The projective special unitary group $\PSU(d)$ acts on $\Irr_{\atom}(d)$ by conjugation: $u\cdot (A,B):=(uAu^*,uBu^*).$
By definition, the $\PSU(d)$-orbits are the irreducible classes of atomic P-modules.  
Moreover, the P-modules $m_{w,\varphi}$ of Section \ref{subsec:atomic-pmod-model} are representatives of these orbits. 
Using Corollary \ref{cor:intro-classif-atom-general} we deduce that $\PSU(d)$ classifies the associated P-representations. Using obvious manifold structures we deduce moduli spaces of atomic P-representations (see Section \ref{sec:geometry} for details).

\begin{lettercor}\label{lettercor:geometry}
Consider $d\geq 1$ and $X=F,T,V,\cO$. We have:
\begin{enumerate}[i]
\item $\PSU(d)\times W_d\times S_1\to \Irr_{\atom}(d),\ (u,A,B)\mapsto u\cdot (A,B)$ is a bijection. Hence, $\Irr_{\atom}(d)$ has an obvious structure of compact smooth real manifold of dimension $d^2$.
\item For $(X,d) \neq (F,1)$ the set of irreducible classes of atomic representations of $X$ with P-dimension $d$ is in bijection with $W_d \times S_1$ (a finite disjoint union of circles).
\item If $d\neq \ti d$ then $\Pi^X(m)\not\cong \Pi^X(\ti m)$ where $m\in\Irr_{\atom}(d)$ and $\ti m\in \Irr_{\atom}(\ti d).$ 
\end{enumerate}
\end{lettercor}

\underline{\textbf{Atomic representations of Dutkay, Haussermann and Jorgensen.}} 
The family of so-called \textit{purely atomic} representations of the Cuntz algebra $\cO$ was defined in \cite{Dutkay-Haussermann-Jorgensen15}.
They coincide with the atomic representations of $\cO$ considered in the current article. Similarly to the current article, in \cite{Dutkay-Haussermann-Jorgensen15} the authors classified the irreducible classes of purely atomic representations of $\cO$. Despite sharing some common features, these two studies are rather different in nature.
Notably, the classification of purely atomic representations is accomplished by studying directly the larger infinite-dimensional Hilbert space $\scrH$ via certain projections associated to singleton sets. In contrast, our study is primarily focused on classifying atomic representations by only studying the smaller finite-dimensional Hilbert space $\fH$. This allows us to follow somewhat simpler arguments and phrase many of our results in terms of finite-dimensional liner algebra.
Furthermore, we are largely concerned on studying the restriction of the atomic representations to the Thompson groups and to explicitly describe them (as direct sum of monomial representations). 
This takes up the majority of the last section in the paper, where else irreducibility and equivalence follows rather easily from classical results. 
Finally, our novel approach of decomposing finite-dimensional P-modules provides an important framework for future work. This will permit us to recover and extend, among other, results appearing in 
\cite{Aita-Bergmann-Conti97, Araujo-Pinto22, Barata-Pinto19, Bergmann-Conti03, Bratteli-Jorgensen19, Dutkay-Haussermann-Jorgensen15, Garncarek12, Guimaraes-Pinto22, Jones21, Kawamura05, Mori-Suzuki-Watatani07, Olesen16}. 
This will be extensively explained in \cite{Brothier-Wijesena23}.

%%%%%%%%%%%%%%%%PRELIMINARIES%%%%%%%%%%%%%%%%%%%%%%%%%%%%%%%%%
%%%%%%%%%%%%%%%%%%%%%%%%%%%%%%%%%%%%%%%%%%%%%%%%%%%%%%%%%%

\section{Preliminaries}\label{sec:preliminaries}
In this section we fix notations (similar to the ones in \cite{Brothier-Wijesena22}) and recall some standard definitions and results.

{\bf Convention.} 
We assume that all groups are discrete, all Hilbert spaces are over the complex field $\C$ with inner-products linear in the first variable, and all group representations are {\it unitary}.
The set of natural numbers $\N$ contains $0$ and we write $\N^*$ for $\N\setminus\{0\}$.

\subsection{Monomial representations and the Mackey--Shoda criterion}

%{\bf Induced and monomial representations.}
If $H\subset G$ is a subgroup and $\sigma$ a representation of $H$, then $\Ind_H^G \sigma$ denotes the \textit{induced representation} of $\sigma$ associated to $H$. 
When $\sigma=\chi$ is one-dimensional (i.e.~valued in the circle group $S_1$) then $\Ind_H^G \chi$ is called \textit{monomial}.
If $\chi=1_H$ is the trivial representation then $\lambda_{G/H}:=\Ind_H^G\chi$ is the \textit{quasi-regular} representation.

{\bf Commensurator.}
Let $H\subset G$ be a subgroup.
\begin{itemize}
	\item The \textit{commensurator} of $H\subset G$ is the subgroup $\Comm_G(H)\subset G$ of $g\in G$ satisfying that $H\cap g^{-1}Hg$ has finite index in both $H$ and $g^{-1}Hg$.
	\item The subgroup $H\subset G$ is \textit{self-commensurating} if $H=\Comm_G(H)$.
\end{itemize}

We recall the celebrated Mackey--Shoda criterion \cite{Mackey51} (see also \cite{Bekka-Harpe20} Theorem 1.F.11, Theorem 1.F.16 and Corollary 1.F.18).

\begin{theorem} \label{Mackey--Shoda criteria}
Let $H_i\subset G_i,i=1,2$ be two subgroups and take one-dimensional representations $\chi_i:H_i\to S_1, i=1,2.$ Set $(H,G):=(H_1,G_1)$. We have the following.
	\begin{enumerate}[i]
		\item The induced representation $\Ind_H^G(\chi)$ is irreducible if and only if for every $g\in\Comm_G(H)$ with $g\notin H$, the restrictions of $\chi:H\to S_1$ and $\chi^g:g^{-1}Hg\ni s\mapsto \chi(gsg^{-1})$ to the subgroup $H\cap g^{-1} H g$ do not coincide.
		\item The induced representation $\Ind_{H_1}^G\chi_1$ is unitary equivalent to $\Ind_{H_2}^G\chi_2$ if and only if there exists $g\in G$ such that $H_1 \cap g^{-1}H_2g$ has finite index in both groups $H_1$ and $g^{-1}H_2g$; and moreover the restrictions of $\chi_2^g$ and $\chi_1$ to $H_1 \cap g^{-1}H_2g$ coincide.
	\end{enumerate}
\end{theorem}
In particular, all monomial representations constructed from self-commensurated subgroups are irreducible. 
Moreover, if $H_1,H_2$ are self-commensurating subgroups, then $\Ind_{H_1}^G\chi_1\cong \Ind_{H_2}^G\chi_2$ if and only if $g^{-1}H_2g=H_1$ and $\chi_2^g=\chi_1$ for some $g\in G$.
Here, unitary equivalent, denoted $\cong$, means that there exits a unitary transformation that intertwinnes the actions.

\subsection{Richard Thompson's groups $F \subset T \subset V$} \label{subsec:F-def}
We refer to \cite{Brown87,Cannon-Floyd-Parry96,Belk04} for details on the Richard Thompson groups and their diagrammatic descriptions.
%%%%%%%%%%%%%%%%%%%%%%%%%%%%%%%%%%%%%%%%%%%%%%%%%%%

{\bf Cantor space and Thompson's groups.}
Let $\cC:=\{0,1\}^{\N^*}$ be the Cantor space of infinite binary strings (also called sequences or rays).
Sequences are written from left to right. Finite binary strings or sequences are called {\it words}. We write $w\cdot u$ for the concatenation of $w$ with $u$.
We equipped these sequences with the lexicographic order.
If $w$ is a finite binary string (also called a word) we form $I_w:=w\cdot \cC$: the set of all sequences with prefix $w$. We call $I_w$ a {\it standard dyadic interval} (sdi in short). 
The terminology comes from the real interval $[0,1]$: the usual surjection $S:\cC\to [0,1],\ x\mapsto \sum_n 2^{-n} x_n$ maps $I_w$ into an interval of the form $[2^{-n}k, 2^{n}(k+1)].$
If $P:=(w_1,\dots,w_n)$ is an $n$-tuple of words such that $(I_{w_1},\dots,I_{w_n})$ forms a partition of $\cC$, then $P$ is called a {\it standard dyadic partition} (sdp in short).
We say that $P$ is oriented if $w_i<w_{i+1}$ for $1\leq i\leq n-1$ for the lexicographic order.
Two sdp's $P=(u_1,\dots,u_n)$ and $Q=(v_1,\dots,v_n)$ with the same number of sdi's defines a homeomorphism $g$ of $\cC$ such as $u_i\cdot x\mapsto v_i\cdot x.$
The collection of all of those form {\it Thompson's group} $V$.
Let $F\subset V$ (resp.~$T\subset V$) to be the subset of $g$ as above where $P$ and $Q$ are oriented (resp.~oriented up to cyclic permutation).
The subsets $F,T$ of $V$ are groups called the {\it Thompson groups} $F$ and $T$.

{\bf Slope.}
Write $|u|$ for the word-length of a word. 
If $g$ maps $I_u$ onto $I_v$ via $u\cdot x\mapsto v\cdot x$, then we say that $2^{|u|-|v|}$ is the slope of $g$ on $I_u$ and write $g'(u\cdot x)=2^{|u|-|v|}$ for all $x \in \cC$.

{\bf Trees, forests, and rays.}
Consider the infinite binary rooted tree $t_\infty$ that we geometrically identify as a graph in the plane where the root is on top, the root has two adjacent vertices to its bottom left and right, and every other vertex has three adjacent vertices: one above, one to the bottom left, and one to the bottom right. The bottom two vertices are called \textit{immediate children} of the vertex (and further down vertices are called {\it children}).
A pair of edges that have a common vertex is called a \textit{caret} and denoted by the symbol $\Y$.
We will constantly identify vertices of $t_\infty$ with words and boundary of $t_\infty$ with sequences in the usual manner (hence the digits $0$ and $1$ correspond to left-edge and right-edge, respectively, and the trivial word corresponds to the root of $t_\infty$).
Elements of the boundary of $t_\infty$ are called {\it rays}.
If $p$ is a ray and $n\geq 0$, then $p_n$ is the word made of the first $n$ digits of $p$ and ${}_np$ the rest of $p$ so that $p=p_n\cdot {}_np.$
If $p=v\cdot w^{\infty}$ for some word $v,w$, then we say that $p$ is {\it eventually periodic} and that $w$ is a {\it period} of $p$. Otherwise we say that $p$ is {\it aperiodic}.
We say that $w$ is {\it prime} when $w\neq u^n$ for all words $u$ and $n\geq 2$.
When $p=v\cdot w^{\infty}$ and $w$ is prime, then we write $|p|$ for the length $|w|$.
If $v$ is trivial, then $p=w^\infty$ and we say that $p$ is {\it periodic}.
Finally, we write $p\sim q$ when $p=q$ modulo finite prefixes, i.e.~there exists $n,k\geq 0$ such that ${}_np={}_kq.$
We write $[p]$ for the class of $p$ for $\sim$.

The term {\it tree} refers to any finite non-empty rooted sub-tree of $t_\infty$ whose each vertex has either none or two immediate children.
They form the set $\cT$. 
If $t$ is a tree, then the vertices of $t$ with no children are called {\it leaves}.
The relation ``being a rooted sub-tree" defines a partial order $\leq$ on $\cT$ for which $\cT$ is directed.
If $s\leq t$, then the diagram obtained by removing $s$ from $t$ is called a {\it forest}.
A forest $f$ is interpreted as a finite ordered list of trees $(f_1,\dots,f_n)$ that has $n$ roots.
Hence, $t$ is obtained by stacking $s$ on top of $f$. We then write $f\circ s$ for $t$. 
This extends to an associated partially defined binary operation on the set $\cF$ of all forests.
This confers to $\cF$ a structure of a small category.
Now, concatenating list of trees: $((f_1,\dots,f_n),(g_1,\dots,g_m))\mapsto (f_1,\dots,f_n,g_1,\dots,g_m)$ defines an associative binary operation on $\cF$. This is a monoidal product that we denote by $\otimes.$ It corresponds to concatenating {\it horizontally} forests.

{\bf Tree-diagrams for the Thompson groups.}
We now describe the Thompson groups using trees.
Consider $(t,\kappa,s)$ where $t,s$ are trees with same number of leaves, say $n$, and $\kappa$ is a permutation on $\{1,\dots,n\}.$
If the $j$th leaf of $s$ and $t$ correspond to the words $u_j,v_j$, respectively, then the triple defines the map $u_i\cdot x\mapsto v_{\kappa(i)}\cdot x$ for $x \in \cC$.
This is an element of $V$ and all elements of $V$ can be achieved in that way.
It is in $F$ (resp.~$T$) if and only if $\kappa$ is trivial (resp.~cyclic).
Assume for simplicity that $\kappa$ is trivial and write $(t,s)$ for $(t,\id,s).$
If $f$ is a forest composable with $t$, then note that $(f\circ t,f\circ s)$ defines the same element of $F$ as $(t,s)$.
Let $\sim$ be the equivalence relation generated by $(t,s)\sim (f\circ t,f\circ s)$ and write $[t,s]$ for the class of $(t,s).$
The set of these classes admits a group structure, isomorphic to $F$, via the composition $[t,s]\circ [s,r]=[t,r]$ and inverse $[t,s]^{-1}=[s,t].$
Similarly, we can define $T$ and $V$ using classes of triples $[t,\kappa,s].$
We call the triples $(t,\kappa,s)$ tree-diagrams and say that $u_i$ and $v_{\kappa(i)}$ are corresponding leaves of $(t,\kappa,s)$.

{\bf Specific notations.}
The trivial tree (the tree with one leaf equal to its root) is denoted $I$ or $e$.
We write $f_{k,n}:=I^{\ot k-1}\ot \Y\ot I^{\ot n-k}$ for the so-called elementary forest that has $n$ roots, $n+1$ leaves, all of its trees trivial except the $k$th one that is equal to a caret $\Y$.
The complete tree with $2^n$ leaves all at a distance $n$ from the root is denoted $t_n$.
If $f$ is a forest, then $\Root(f)$ and $\Leaf(f)$ denote its root-set and leaf-set, respectively.
We write $\Ver$ for the vertex-set of the rooted infinite complete binary tree $t_\infty$.
If $p\in\cC$, then $\Ver_p$ denotes all the finite prefixes $p_n$ of $p$ (i.e.~the vertices that the ray $p$ is passing through).

%%%%%%%%%%%%%%%%%%%%%%%%%%%%%%%%%%%%%%%%%%%%%%%%%%%%%%

\subsection{Parabolic subgroups of the Thompson groups}
For this subsection we shall take $X$ to denote any of $F,T,V$. 
Given $p\in\cC$ we form the so-called {\it parabolic subgroups}
$$X_p:=\{g\in X:\ g(p)=p\}.$$
Note that $X_p\neq X$ except when $X=F$ and $p$ is an endpoint of $\cC$. 

\subsubsection{Description of parabolic subgroups using tree-diagrams} \label{parabolic desc subsection}
Consider a ray $p$ with $n$th digit $x_n$.
Given $g=[t,\kappa,s]\in X$ we have that there exists a unique leaf $\nu$ of $t$ and $\omega$ of $s$ so that $\nu,\omega$ lie in the ray $p$ (equivalently $p\in I_\nu\cap I_\omega$).
If $|\nu|=m$ and $|\omega|=n$, then $\nu=p_m$ and $\omega=p_n$.
By definition of the action $V\act \cC$ we have that $g \in X_p$ if and only if $p_m$ and $p_n$ are corresponding leaves of $g$ and ${}_mp = {}_np$. 

Assume that $p$ is eventually periodic such that $p=v \cdot w^\infty$. Take $v$ as small as possible.
Then from the preceding paragraph we have
\begin{equation}\label{parabolic subgroup condition eqn}
	\begin{cases*}
		m = n, &\quad \text{ if } $m,n \leq |v|$ \\
		n-m \in |w|\Z, &\quad \text{ if } $m,n > |v|$.
	\end{cases*}
\end{equation} 
In the case when $p$ is {\it not eventually periodic} we obtain that $X_p$ is the group elements acting trivially on a neighbourhood of $p$  (i.e.~have slope $1$ at $p$).

\subsubsection{Monomial representations associated to parabolic subgroups} \label{general monomial rep section}
It is standard that that $X_p\subset X$ is self-commensurating.
Then Mackey--Schoda implies the following.

\begin{lemma} \label{lem:monomial-rep-class}
	Consider rays $p_i$ and one-dimensional representations $\chi_i$ of $X_{p_i}$ for $i=1,2$. 
	\begin{enumerate}[i]
		\item The monomial representation $\Ind_{X_{p_1}}^X \chi_1$ of $X$ is irreducible.
		\item If $(X,\vert p_i \vert) \neq (F,1)$ then $\Ind_{X_{p_1}}^X \chi_1\simeq \Ind_{X_{p_2}}^X \chi_2$ if and only if $\chi_1 \cong \chi_2$, and $[p_1] = [p_2]$.
	\end{enumerate}
\end{lemma}

\subsection{Universal $C^*$-algebras} \label{sec:cuntz-algebra}
\textbf{The Cuntz algebra.}
The Cuntz algebra $\cO := \cO_2$ is the universal $C^*$-algebra with two generators $s_0,s_1$ satisfying the below relations:
$$s_0^*s_0 = s_1^*s_1 = s_0s_0^* + s_1s_1^* = 1.$$
Thus, any representation of $\cO$ on a Hilbert space $\scrH$ is given by two isometries, $S_0$ and $S_1$, with orthogonal ranges that span $\scrH$. 
If $\nu$ is a binary word in $0$'s and $1$'s, then denote $s_\nu$ to be the corresponding composition respecting the order of the digits (hence if $\nu = 01$ then $s_\nu = s_0s_1$).
Birget and Nekrashevyvch independently made the remarkable discovery that Thompson's group $V$ embeds inside the unitary group of the Cuntz algebra $\cU(\cO)$ \cite{Birget04, Nekrashevych04}. Indeed, take $g \in V$ and let $\{\nu_i\}_{i=1}^n$, $\{\omega_i\}_{i=1}^n$ be vertices in $\Ver$ such that $g(\omega_i \cdot x) = \nu_{\kappa(i)} \cdot x$ for $x \in \cC$ and some permutation $\kappa$.
The formula
\[V \ni g \mapsto \sum_{i=1}^ns_{\nu_i}s^*_{\omega_i} \in \cO\]
defines an embedding of $V$ into $\cU(\cO)$. 
In fact, $V$ corresponds to the normaliser subgroup of the diagonal sub-algebra $\cA$ inside $\cO$ (where $\cA$ is generated by all the projections $s_{\nu}s_{\nu}^*$).
\begin{center}{\bf From now on we identify $V$ and its image inside $\cO$.}\end{center}
Consequently, every representation of $\cO$ restricts to a (unitary) representation of $V$.

\textbf{The Pythagorean algebra.}
The Pythagorean algebra is the universal $C^*$-algebra $\cP$ with two generators $a$,$b$ satisfying the Pythagorean equality:
\[a^*a+b^*b = 1.\]
Hence, $a\mapsto s_0^*,b\mapsto s_1^*$ defines a surjective *-morphism $\cP\onto\cO$.
Note that $\cP$ has many (non-zero) finite-dimensional representations while $\cO$ has none.

\subsection{Pythagorean representations} \label{sec:def-pyth}
We introduce the specific class of Jones' representations that we will focus on.

\subsubsection{Pythagorean module}\label{sec:P-module}
Here is the main concept of our study.

\begin{definition} \label{universal pythag algebra definition}
	A \emph{Pythagorean module} (in short P-module) is a triple $m=(A,B,\fH)$ where $\fH$ is a Hilbert space and $A,B \in B(\fH)$ are bounded linear operators satisfying the so-called Pythagorean equality
	\[A^*A+B^*B = \id_{\fH}\]
	where $\id_\fH$ is the identity operator of $\fH$ and $A^*$ is the adjoint of $A$. For convenience, we may interchangeably refer to $\fH$ and $m$ as being a P-module.
\end{definition}

For P-modules $m = (A,B,\fH),\ \ti m = (\ti A, \ti B, \ti \fH)$ we say that:
\begin{itemize}
	\item $\fK \subset \fH$ defines a \textit{sub-module} if $\fK$ is closed under $A$ and $B$, in which case we equip $\fK$ with the P-module structure obtained by taking the restrictions of $A$ and $B$;
	\item $m$ is \textit{irreducible} if $m$ does not contain any proper non-trivial sub-modules;
	\item $\theta : \fH \rightarrow \ti \fH$ is an \textit{intertwinner} or \textit{morphism} between $m$ and $\ti m$ if it is a bounded linear operator satisfying
	\[\theta \circ A = \ti A \circ \theta \textrm{ and } \theta \circ B = \ti B \circ \theta;\]
	\item $m$ and $\ti m$ are \textit{unitarily equivalent} (we often drop the term ``unitarily'') if there exists a unitary intertwinner between them. In that case we write $m\cong \ti m$;
	\item $\fK \subset \fH$ is a \textit{complete} sub-module of $m$ if the orthogonal complement $\fK^\perp$ of the subspace $\fH$ inside $\fH$ does not contain any non-trivial sub-modules;
	\item $m$ is a \textit{full} sub-module if $m$ does not contain any proper complete sub-modules;
	\item $\fZ \subset \fH$ is a \textit{residual} subspace if $\fZ$ does not contain any non-trivial sub-modules and $\fZ^\perp$ is a sub-module of $m$.
\end{itemize}

\begin{remark} The P-modules and their morphisms form a category denoted $\Mod(\cP)$. 
Let $\Rep(\cP)$ denote the usual category of representations of the $C^*$-algebra $\cP$. Then $\Mod(\cP)$ and $\Rep(\cP)$ have same class of objects. However, there are more morphisms in $\Mod(\cP)$ than there are in $\Rep(\cP)$. 
Although, it can be proven that a morphism between two {\it full} P-module is in fact a morphism of the associated representations.
\end{remark}

\subsubsection{From a P-module to a Hilbert space}
Fix a P-module $(A,B,\fH)$. 
For each tree $t$ with $n$ leaves we consider the Hilbert space $\fH_t:=\fH^{\Leaf(t)}$ of all maps from the leaves-set of $t$ to $\fH$. We identify $\fH_t$ with $\fH^{\oplus n}$ and write $(t,\xi)$ for an element of it. We may write $\xi_\ell$ or $\xi_i$ for the component corresponding to the leaf $\ell$ of $t$ or to the $i$th leaf.
For each forest $f$ with $n$ roots and $m$ leaves we have an isometry 
$$\Phi(f):\fH_t\to \fH_{ft}$$
obtained by placing the operator $R:=A\oplus B:\fH\to \fH\oplus\fH$ at each node of $f$. 
For instance $\Phi(I\otimes \Y)(\Y,\xi)=((I\ot \Y)\Y, \xi_1, A\xi_2,B\xi_2).$
This defines a functor $\Phi:\cF\to\Hilb$ from the category of binary forests to the category of Hilbert spaces. 
It is monoidal for the horizontal concatenation of forests and the direct sum of Hilbert spaces.

This forms a directed system of Hilbert spaces, indexed by the directed set of trees $\fT$, with maps being the $\Phi(f)$.
The limit is a pre-Hilbert space $\scrK:=\varinjlim_{t\in\fT}\fH_t$ that we complete into $\scrH$. (It indeed has an inner-product because all the $\Phi(f)$ are isometric.)
Equivalently, $\scrH$ is the disjoint union of the $\fH_t$ modulo the equivalence relation generated by $(t,\xi)\sim (ft,\Phi(f)\xi).$
We write $[t,\xi]$ for the class of $(t,\xi)$ inside $\scrH$ and note that $(t,\xi)\mapsto [t,\xi]$ defines an isometric embedding $\fH_t\into \scrH$. 
Moreover, $\xi\mapsto (I,\xi)$ defines an isomorphism $\fH\simeq \fH_I$. 
We will often identify $\fH$ with $\fH_I$ and $\fH_t$ with its image inside $\scrH$.
We note that $\dim(\scrH)=\infty$ (except when $\fH=\{0\}$).

\subsubsection{Partial Isometries on $\scrH$} \label{subsec:partial-isom}

Fix $\nu \in \Ver$ (i.e.~$\nu$ is a finite binary sequence) and consider $[t, \xi] \in \scrK$. Up to growing $t$ we can assume that $\nu$ is a vertex of $t$. Define $t_\nu$ to be the sub-tree of $t$ with root $\nu$ and whose leaves are the leaves of $t$ which are children of $\nu$. Then we set $\tau_\nu([t, \xi]) := [t_\nu, \eta]$ where $\eta$ is the decoration of the leaves of $t$ in $[t, \xi]$ that are children of $\nu$. It can be shown that $\tau_\nu$ is well-defined and extends to a surjective partial isometry from $\scrH$ onto itself 
(see \cite[Section 2.1]{Brothier-Wijesena22} for details).

We will be considering projections $\rho_\nu:=\tau_\nu^*\tau_\nu$.
More generally, consider $\tau^*_\nu\tau_\omega$ the partial isometry which ``snips'' the tree at $\omega$ and attaches the resulting sub-tree, along with its components, at the vertex $\nu$ while setting all other components to $0$. Here is an example:

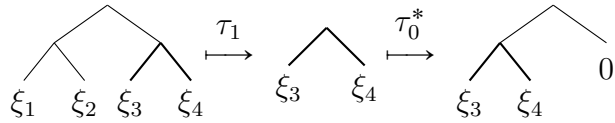
\begin{center}
	\begin{tikzpicture}[baseline=0cm, scale = 1]
		\draw (0,0)--(-.7, -.5);
		\draw (0,0)--(.7, -.5);
		\draw (-.7, -.5)--(-1.1, -1);
		\draw (-.7, -.5)--(-.3, -1);
		\draw[thick] (.7, -.5)--(.3, -1);
		\draw[thick] (.7, -.5)--(1.1, -1);
		
		\node[label={[yshift=-22pt] \normalsize $\xi_1$}] at (-1.1, -1) {};
		\node[label={[yshift=-22pt] \normalsize $\xi_2$}] at (-.3, -1) {};
		\node[label={[yshift=-22pt] \normalsize $\xi_3$}] at (.3, -1) {};
		\node[label={[yshift=-22pt] \normalsize $\xi_4$}] at (1.1, -1) {};
		
		\node[label={[yshift= -3pt] \normalsize $\tau_{1}$}] at (1.6, -.7) {$\longmapsto$};
	\end{tikzpicture}%
	\begin{tikzpicture}[baseline=0cm, scale = 1]
		\draw[thick] (0,-.3)--(-.5, -.8);
		\draw[thick] (0,-.3)--(.5, -.8);
		
		\node[label={[yshift=-22pt] \normalsize $\xi_3$}] at (-.5, -.8) {};
		\node[label={[yshift=-22pt] \normalsize $\xi_4$}] at (.5, -.8) {};
		
		\node[label={[yshift= -3pt] \normalsize $\tau^*_{0}$}] at (1.1, -.7) {$\longmapsto$};
	\end{tikzpicture}%
	\begin{tikzpicture}[baseline=0cm, scale = 1]
		\draw (0,0)--(-.7, -.5);
		\draw (0,0)--(.7, -.5);
		\draw[thick] (-.7, -.5)--(-1.1, -1);
		\draw[thick] (-.7, -.5)--(-.3, -1);
		
		\node[label={[yshift=-22pt] \normalsize $\xi_3$}] at (-1.1, -1) {};
		\node[label={[yshift=-22pt] \normalsize $\xi_4$}] at (-.3, -1) {};
		\node[label={[yshift=-22pt] \normalsize $0$}] at (.7, -.5) {};		
	\end{tikzpicture}%
\end{center}

\begin{observation}
Note that $(\tau_0, \tau_1, \scrH)$ forms a P-module. Moreover, $\tau_0\restriction_\fH = A$ and $\tau_1\restriction_\fH = B$. Hence, $(A,B,\fH)$ is a sub-module of $(\tau_0, \tau_1, \scrH)$. 
\end{observation}

\subsubsection{Pythagorean representations from P-modules}
Consider as above our fixed P-module $m=(A,B,\fH)$, the functor $\Phi:\cF\to\Hilb$, and the Hilbert space $\scrH$.
Take $g\in F$ and $[r,\xi]\in \scrH$.
There exists some trees $t,s$ such that $g=[t,s]$.
Now, there exists forests $f,h$ such that $fs=hr$.
We set $$g\cdot [r,\xi]:=[ft,\Phi(h)\xi].$$
In particular, if $s=r$, then $[t,s]\cdot [s,\xi]:=[t,\xi].$
This defines a unitary representation $\sigma:F\act\scrH$ called the Pythagorean representation (in short P-representation) associated to the P-module $m$.
Now, if $v\in V$, then $v=[t,\kappa,s]$ for some permutation $\kappa$.
We set 
$v\cdot [s,\xi]:=[t,\xi_{\kappa}]$ where $\xi_{\kappa}(i):=\xi_{\kappa(i)}$ (the $i$th coordinate of $\xi_{\kappa}$ is the $\kappa(i)$th coordinate of $\xi$).
This define a unitary representation $\sigma^T,\sigma^V$ of $T,V$ on the same Hilbert space $\scrH$.
Using the partial isometries of the previous subsection we deduce the following formula:
\begin{equation} \label{action rearrange equation}
	\sigma(v) = \sum_{i = 1}^n \tau^*_{\nu_{\kappa(i)}} \tau_{\omega_i},
\end{equation}
where $\nu_i,\omega_i$ are the $i$th leaves of $t,s$, respectively, and where $t$ has $n$ leaves.

\begin{example}An example of the action $\sigma$ is shown below. In this case, $\sigma$ only changes the tree while retaining the original decoration.

\begin{center}
	\begin{tikzpicture}[baseline=0cm]
		\draw (0,0)--(-.5, -.5);
		\draw (0,0)--(.5, -.5);
		\draw (-.5, -.5)--(-.9, -1);
		\draw (-.5, -.5)--(-.1, -1);
		
		\node[label={\normalsize $\sigma($}] at (-1.1, -1) {};
		\node[label={\normalsize $,$}] at (.65, -1) {};
	\end{tikzpicture}%
	\begin{tikzpicture}[baseline=0cm]
		\draw (0,0)--(-.5, -.5);
		\draw (0,0)--(.5, -.5);
		\draw (.5, -.5)--(.1, -1);
		\draw (.5, -.5)--(.9, -1);
		
		\node[label={\normalsize $)$}] at (1.1, -1) {};
		\node[label={\normalsize $\cdot$}] at (1.35, -.9) {};
	\end{tikzpicture}%
	\begin{tikzpicture}[baseline=0cm]
		\draw (0,0)--(-.5, -.5);
		\draw (0,0)--(.5, -.5);
		\draw (.5, -.5)--(.1, -1);
		\draw (.5, -.5)--(.9, -1);
		
		\node[label={[yshift=-22pt] \normalsize $\xi_1$}] at (-.5, -.5) {};
		\node[label={[yshift=-22pt] \normalsize $\xi_2$}] at (.1, -1) {};
		\node[label={[yshift=-22pt] \normalsize $\xi_3$}] at (.9, -1) {};	
		
		\node[label={\normalsize $=$}] at (1.35, -1) {};
	\end{tikzpicture}%
	\begin{tikzpicture}[baseline=0cm]
		\draw (0,0)--(-.5, -.5);
		\draw (0,0)--(.5, -.5);
		\draw (-.5, -.5)--(-.9, -1);
		\draw (-.5, -.5)--(-.1, -1);
		
		\node[label={[yshift=-22pt] \normalsize $\xi_1$}] at (-.9, -1) {};
		\node[label={[yshift=-22pt] \normalsize $\xi_2$}] at (-.1, -1) {};
		\node[label={[yshift=-22pt] \normalsize $\xi_3$}] at (.5, -.5) {};		
	\end{tikzpicture}%
\end{center}
\end{example}

\textbf{Extension to $\cO$.}
Using the partial isometries $\tau_\nu$ we can easily observe that $\sigma^V$ extends to a representation $\sigma^\cO$ of $\cO$ via the formula
\[\sigma^\cO : \cO \rightarrow B(\scrH),\ s_0 \mapsto \tau_0^*,\ s_1 \mapsto \tau_1^*.\]
Using \eqref{action rearrange equation} we deduce that the representation $\sigma^\cO$ restricted to $V$ is equal to $\sigma^V$.
Surprisingly, every representation of $\cO$ can be obtained in this manner (\cite[Proposition 7.1]{Brothier-Jones19a}). 
Indeed, if $(\pi,\scrH)\in \Rep(\cO)$, then $m=(\pi(s_0)^*,\pi(s_1)^*,\scrH)$ is a P-module. If we re-apply the construction of above to $m$ we obtain a new representation $\sigma^\cO$ of $\cO$ which is equivalent to $\pi$.

\subsubsection{Sub-modules and sub-representations}
Sub-modules of $\scrH$ define sub-representations of $\sigma$ as explained below.

\begin{definition}
	Let $\scrX \subset \scrH$ be a sub-module (i.e.~$\tau_i(\scrX)\subset \scrX$ for $i=0,1$). Define the closed subspace
	\[\sub{\scrX} := \cspan\{\cup_{\nu \in \Ver} \tau_\nu^*(\scrX)\} \subset \scrH.\]
	By construction $\scrX$ is closed under the action of $\sigma^\cO$ and defines a sub-representation denoted $\sigma^\cO_\scrX$. Similarly, we define $\sigma^Y_\scrX$ for $Y=F,T,V$.
\end{definition}

Informally, $\sub{\scrX}$ is the closure of the set of trees with leaves decorated by vectors in $\scrX$.

\begin{observation}
If $\scrX$ is a sub-module of $\scrH$, then the P-representation associated to the P-module $(\tau_0\restriction_\scrX, \tau_1\restriction_\scrX, \scrX)$ is equivalent to $\sigma_\scrX$. Hence, all sub-representations that are induced by sub-modules are also Pythagorean. For the remainder of the article we will freely identify these two representations.
\end{observation}

\subsubsection{Functoriality of Jones' technology} \label{subsec:p-functor}
Recall that $\Mod(\cP)$ denotes the category of P-modules with morphisms being bounded linear maps intertwinning the $A$'s and $B$'s. Moreover, $\Rep(X)$ denotes the usual category of unitary representations when $X=F,T,V$ and the usual category of bounded linear *-representations when $X=\cO$.

Let $m = (A,B,\fH)$ and $\ti m = (\ti A, \ti B, \ti \fH)$ be two P-modules with associated P-representations $(\sigma, \scrH)$ and $(\ti\sigma, \ti\scrH)$, and functors $\Phi,\ti\Phi$, respectively. Let $\theta : \fH \rightarrow \ti \fH$ be an intertwinner between the two P-modules. For any tree $t$, this gives a map $\theta_t : \fH_t \rightarrow \ti\fH_t$ by
\[\theta_t : (t, \xi_1, \xi_2, \dots, \xi_n) \mapsto (t, \theta(\xi_1), \theta(\xi_2), \dots, \theta(\xi_n))\]
where $n$ is the number of leaves of $t$. Diagrammatically, $\theta_t$ works as follows when $t=\Y$:
\begin{center}
	\begin{tikzpicture}[baseline=0cm, scale = 1]
		\draw (0,0)--(-.5, -.5);
		\draw (0,0)--(.5, -.5);
		
		\node[label={[yshift=-22pt] \normalsize $\xi_1$}] at (-.5, -.5) {};
		\node[label={[yshift=-22pt] \normalsize $\xi_2$}] at (.5, -.5) {};
		
		\node[label={[yshift= -3pt] \normalsize $\theta_{\Y}$}] at (1.4, -.7) {$\longmapsto$};
	\end{tikzpicture}%
	\begin{tikzpicture}[baseline=0cm, scale = 1]
		\draw (0,0)--(-.5, -.5);
		\draw (0,0)--(.5, -.5);
		
		\node[label={[yshift=-22pt] \normalsize $\theta(\xi_1)$}] at (-.55, -.5) {};
		\node[label={[yshift=-22pt] \normalsize $\theta(\xi_2)$}] at (.55, -.5) {};		
	\end{tikzpicture}.
\end{center}
Since $\theta$ is an intertwinner this implies
\begin{equation} \label{eqn:intertwinner-direct-system}
	\theta_t \circ \Phi(f) = \ti\Phi(f) \circ \theta_t
\end{equation}
for any composable forest $f$. 
From there we deduce a bounded linear map 
$$\Theta:\scrH\to \ti\scrH,\ [t,\xi]\mapsto [t,\theta_t(\xi)]$$
that intertwinnes the P-representations.
We deduce four functors that we name the {\it Pythagorean functors}:
$$\Pi^X:\Mod(\cP)\to \Rep(X).$$
\textbf{Notation.} From a P-module $m=(A,B,\fH)$ we have canonically constructed the P-representations $\sigma^X=\Pi^X(m)$, for $X = F,T,V,\cO$, all acting on the same Hilbert space $\scrH$. 
We may drop the super-script $X$ if it is clear from context or when 
making statements that hold true for all $X$. Additionally, despite all representations of $\cO$ coming from P-modules, we may term a representation of $\cO$ as being ``Pythagorean'' to emphasise we are viewing it as arising from a P-module

\begin{center}\textbf{For the remainder of the paper, we shall assume that all P-modules are \emph{finite-dimensional}. }
\end{center}

%%%%%%%%%%%%%%%%%END PRELIMINARIES%%%%%%%%%%%%%%%%%%%%%%%%%%%%%%
%%%%%%%%%%%%%%%%%%%%%%%%%%%%%%%%%%%%%%%%%%%%%%%%%%%%%%%%%%%

%%%%%%%%%%%%%%%%%%%%%%%%%%%%%%%%%%%%%%%%%%%%%%%%%%%%%%%%%%%
%%%%%%%%%%%%%%%%%%%STRUCTURE%%%%%%%%%%%%%%%%%%%%%%%%%%%%%

\section{Decomposition of P-modules} \label{sec:p-mod-decomp}
In this section we will introduce the important notions of the ``atomic'' and ``diffuse'' parts of a P-module. Furthermore, we will develop a powerful decomposition theory. 

\subsection{Complete sub-modules.}
It should be emphasised that sub-modules of P-modules are only required to be closed under $A$ and $B$, but not necessarily under $A^*$ or $B^*$. 
Hence, the orthogonal complement of a sub-module may not be a sub-module making $\Mod(\cP)$ not semi-simple (unlike $\Rep(\cP)$).
However, a weaker property holds as shown below. 
We refer to Section \ref{sec:P-module} for the definition of residual and complete.

\begin{lemma} \label{lem:p-mod-irr-decomp}
	Let $\fH$ be a P-module. 
There exists $n\geq 1$, some irreducible sub-modules $\fH_i\subset \fH, 1\leq i\leq n$, and a vector subspace $\fH_{\res}\subset\fH$ such that $\fH=(\oplus_{i=1}^n \fH_i)	\oplus \fH_{\res}.$
We call $\fH_{\res}$ the \emph{residual subspace} of $\fH$ and $\fH_{\comp}:=\oplus_{i=1}^n\fH_i$ the \emph{complete sub-module} of $\fH$.
\end{lemma}
	
\begin{proof}[Proof of the lemma]
	Consider a P-module $\fH$. If $\fH$ is not irreducible, then there exists an irreducible sub-module $\fK \subset \fH$ (since $\fH$ is finite-dimensional). If $\fK^\perp$ is an irreducible sub-module or a residual subspace we are done. Otherwise, $\fK^\perp$ contains a proper non-trivial sub-module and iteratively repeat the above process which must eventually terminate since $\fH$ is finite-dimensional.
\end{proof}

\begin{example}
Here is an example where the decompositions in $\Mod(\cP)$ is thinner than in $\Rep(\cP).$
Consider
	\[m = ( 
	\begin{pmatrix}
		\frac{1}{\sqrt{2}} & 0 & \frac{1}{\sqrt{6}} \\
		0 & -\frac{1}{\sqrt{2}} & \frac{1}{\sqrt{6}} \\
		0 & 0 & \frac{1}{\sqrt{6}}
	\end{pmatrix},
	\begin{pmatrix}
		\frac{1}{\sqrt{2}} & 0 & -\frac{1}{\sqrt{6}} \\
		0 & -\frac{1}{\sqrt{2}} & -\frac{1}{\sqrt{6}} \\
		0 & 0 & \frac{1}{\sqrt{6}}
	\end{pmatrix},
	\C^3).\]
	Let $e_1,e_2,e_3$ be the standard basis elements of $\C^3$. We have $\fH_{\comp} = \C e_1 \oplus \C e_2$ which decomposes into irreducible sub-modules $(1\sqrt{2},1\sqrt{2},\C) \oplus (-1\sqrt{2},-1\sqrt{2},\C)$, and $\fH_{\res} = \C e_3$. Yet, it is easy to verify that the associated $^*$-representation of $\cP$ on $\C^3$ is irreducible.
\end{example}

\begin{remark}
	If $\fH$ is full (that is, $\fH_{\comp} = \fH$ and $\fH_{\res} = \{0\}$) then the decomposition with respect to the action of the P-module and the action of the Pythagorean algebra $\cP$ do coincide. This is because in a full P-module, every $A,B$-invariant subspace is also a $A^*,B^*$-invariant subspace. 
\end{remark}

\begin{proposition} \label{prop:complete-mod-rep}
	Let $m = (A,B,\fH)$ be a P-module decomposed as in in Lemma \ref{lem:p-mod-irr-decomp} and let $\fK\subset\fH$ be a sub-module. For each $X=F,T,V,\cO$ we have:
	\begin{enumerate}[i]
		\item $\Pi^X(\fK) \cong \Pi^X(\fH)$ if and only if $\fK$ is a complete sub-module;
		\item there exists a sub-module $\fX \subset \fK^\perp \subset \fH$ such that $\Pi^X(\fK) \oplus \Pi^X(\fX) \cong \Pi^X(\fH)$;
		\item $\Pi^X(\fH) \cong \oplus_{i=1}^n \Pi^X(\fH_i).$
	\end{enumerate}
\end{proposition}

\begin{proof}
The last two items immediately follows from the first item and Lemma \ref{lem:p-mod-irr-decomp}. Hence we only need to prove the first item. For the forward implication, let $\fX \subset \fK^\perp$ be a sub-module. Then $\sub{\fX} \subset \sub{\fK}^\perp$. However, $\sub{\fK} = \scrH$ and thus $\sub{\fK}^\perp = \{0\}$. This implies that $\fX$ is the trivial sub-modules and $\fK$ is a complete sub-modules.
	
The converse of item i is equivalent to showing that $\sub{\fH_{\comp}} = \scrH$ as $\fH_{\comp}$ is contained inside every complete sub-module of $\fH$. Since $\sub{\fH} = \scrH$, it is suffice to show that $\fH_{\res} = \fH_{\comp}^\perp \subset \sub{\fH_{\comp}}$.
Heuristically we will proceed as follows: fix $\xi\in\fH_{\res}$ and let $\xi_t$ its representative in $\fH_t$ inside $\scrH$ (i.e.~$\xi_t:=\Phi(t)\xi$).
Then, when $t$ grows we note that the distance between each component of $\xi_t$ and $\fH_{\comp}$ tends to zero. Moreover, we will show that this convergence is ``global'' among all the leaves in such a way that the distance between the whole vector $\xi_t$ and $\sub{\fH_{\comp}}$ is manifestly small for $t$ large. 
This will be achieved by making a compactness argument on the closed unit ball $(\fH_{\res})_1$.
	
Assume $\fH_{\res}$ is non-zero (the zero case being trivially true). 
Let $P:\fH\to\fH$ be the orthogonal projection onto $\fH_{\res}$. 
For each $n$ consider $t_n$ the complete binary with $2^n$ leaves all at distance $n$ from the root and the isometry
$$\varphi_n:\fH\to \fH_{t_n},\ \xi\mapsto \Phi(t_n)\xi.$$
Note that $\fH_{t_n}$ is isomorphic to the $2^n$-fold direct sum of $\fH$ and the map is conjugated to 
$\xi\mapsto \bigoplus_{|w|=n} w\xi$
where $w$ is any word in $A,B$ made of $n$ letters.
Define now
\[\psi_n : \fH_{\res} \rightarrow \R,\ \xi \mapsto \sum_{\vert w \vert = n} \norm{P(w\xi)}^2\]
which is the restriction to $\fH_{\res}$ of the composition of $\varphi_n$, $P\ot \id$, and the norm square.
This is obviously continuous with operator norm smaller or equal to $1$.
Since the space $\fH$, and thus $\fH_{\res}$, is finite-dimensional its closed unit ball $(\fH_{\res})_1$ is compact. Therefore, each of the maps $\psi_n$ restricted to $(\fH_{\res})_1$ attain a maximum value $M_n\leq 1$ at some vector $\xi_n \in (\fH_{\res})_1$.
The remainder of the proof will be separated into individual claims.
	
	\textbf{Claim 1:}
	There exists a natural number $N > 0$ such that $M_N < 1$.
	
	Suppose that $M_n = 1$ for all $n$. Again, by appealing to the compactness of the closed unit ball, there exists a sub-sequence $(\xi_{m_n})_n$ which converges to some $\xi \in (\fH_{\res})_1$ and satisfies $\psi_n(\xi)=1$.
	This implies that necessarily $w\xi$ is in $\fH_{\res}$ for all words $w$. Hence, $\xi$ defines a non-zero sub-module of $\fH_{\res}$ yielding a contradiction.
		
	For the remainder of the proof we fix a unit vector $\xi \in \fH_{\res}$. Since $\fH_{\comp}$ is a sub-module it is clear that $(\psi_{n}(\xi) : n > 0)$ forms a decreasing sequence which is bounded below by $0$. Hence, the sequence limits to some constant $M_\xi \in [0,1)$. 
	
	\textbf{Claim 2.}
	The constant $M_\xi$ is equal to zero.
	
	Assume $M_\xi$ is strictly positive.
	Using Claim 1 we fix $N\geq 1$ such that $M_N<1$.
	There exists $K>0$ 
	such that $\psi_k(\xi) - M _\xi< M_\xi(1-M_N)$ for all $k \geq K$. 
	We can then write for any $j > 0$:
	\[\psi_{j+k}(\xi) = \sum_{|v|=j,|u|=k} \norm{P(vu\xi)}^2 = \sum_{|v|=j,|u|=k}  \norm{Pv(Pu\xi)}^2 = \sum_{|u| = k} \psi_j(P(u\xi))\]
	where the second equality follows by noting that $PvP = Pv$ for all words $v$ because $\fH_{\comp}$ is invariant under $A,B$. Hence we obtain for $k \geq K$:
	\begin{align*}
		\psi_k(\xi) - \psi_{N+k}(\xi) &= \sum_{\vert w \vert = k} \norm{P(w\xi)}^2 - \psi_N(P(w\xi)) \geq \sum_{\vert w \vert = k} \norm{P(w\xi)}^2 - M_N\norm{P(w\xi)}^2 \\
		&= \psi_k(\xi)(1-M_N) \geq M_\xi(1 - M_N). 
	\end{align*}
We have obtained a contradiction.

We are now able to conclude. Indeed, by Claim 2, for each $\epsilon>0$ there exists $k\geq 1$ such that $\psi_k(\xi)<\epsilon.$
Hence, for all $\nu$ of length $k$ there is $\eta_\nu\in\fH_{\comp}$ so that 
\[\norm{\xi - \sum_{\vert \nu \vert = k} \tau^*_\nu(\eta_\nu)}^2 = \psi_k(\xi) < \epsilon.\]
Since $\eta$, the vector of $\fH_{t_n}$ with $\nu$-coordinate being $\tau^*_\nu(\eta_\nu)$, is in $\langle \fH_{\comp}\rangle$ we deduce that the distance $d(\xi,\langle \fH_{\comp}\rangle)$ between $\xi$ and $\langle \fH_{\comp}\rangle$ is smaller than $\epsilon$.
Hence, $d(\xi,\langle \fH_{\comp}\rangle)=0$ and thus $\xi\in\langle \fH_{\comp}\rangle$ since the latter is topologically closed.
\end{proof}

\begin{remark} \label{rem:fH-decomp}
The above proposition motivates the terminology of complete sub-modules and residual subspaces. Indeed, complete sub-modules contain the ``complete'' information of $\Pi(\fH)$ while in contrast residual subspaces are the ``residual'' space left over from complete sub-modules and do not contain any information of $\Pi(\fH)$. 
\end{remark}

\subsection{Classes of vectors in P-modules}
We now introduce two classes of vectors in $\fH$ that will play an important role in our decomposition of Pythagorean representations. Recall we identify rays in $\cC$ as an infinite sequence of $0,1$ which is read from left to right.

\begin{definition} \label{class of vectors definition}
	Fix a P-module $(A,B,\fH)$. Let $p$ be a ray and $\xi$ a vector in $\fH$. 
	For a finite word $w$ in binary digits we write $w\xi$ to denote the action of the operator obtained by replacing the digits of $w$ with $A,B$ and reversing the order (e.g.~if $w = 011$ then $w\xi = BBA\xi$).
	Recall $p_n$ is the first $n$ digits of $p$ for $n \in \N$. If $\xi$ is \emph{non-zero} then we say:
	\begin{itemize}
		\item $\xi$ is \textit{contained} in the ray $p$ if $\lim_{n \to \infty}\norm{p_n\xi} = \norm{\xi}$ (i.e.~$\norm{p_n\xi} = \norm{\xi}$ for all $n \in \N$);
		\item $\xi$ is \textit{annihilated} by all rays if $\lim_{n \to \infty} \norm{q_n\xi} =0$ for all rays $q$.
	\end{itemize}
\end{definition}

The following observation explains the diagrammatic natures of the definitions of above.

\begin{observation} \label{contained vector norm obs}
Consider a ray $p$ and recall that $\Ver_p$ denotes the vertices equal to all finite prefixes $p_n$ of $p$.
A non-zero vector $\xi \in \fH$ is contained in $p$ if and only if $\norm{\tau_{\nu}(\xi)} = \norm{\rho_{\nu}(\xi)} = \norm{\xi}$ for $\nu \in \Ver_p$ and $\tau_\nu(\xi) = \rho_{\nu}(\xi) = 0$ for $\nu \notin \Ver_{p}$. 
\end{observation}

\begin{example}
Consider the P-module 
\[(
\begin{pmatrix}
	0 & 0 \\
	1 & 0
\end{pmatrix},
\begin{pmatrix}
	0 & 1 \\
	0 & 0
\end{pmatrix},
\C^2).\] 
Then $e_1= (1\ 0)^T  \in \fH$ is contained in the zig-zag ray $(01)^\infty=0101\dots$ as shown below.

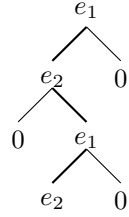
\begin{center}
	\begin{tikzpicture}[baseline=0cm, scale = 0.9]
		\draw[thick] (0,0)--(-.5, -.5);
		\draw (0,0)--(.5, -.5);
		\draw (-.5, -.9)--(-1, -1.4);
		\draw[thick] (-.5, -.9)--(0, -1.4);
		\draw[thick] (0, -1.8)--(-.5, -2.3);
		\draw (0, -1.8)--(.5, -2.3);
		
		\node[label={[yshift=-5pt] \footnotesize $e_1$}] at (0, 0) {};
		\node[label={[yshift=-18pt] \footnotesize $e_2$}] at (-.5, -.5) {};
		\node[label={[yshift=-18pt] \footnotesize $0$}] at (.5, -.5) {};
		\node[label={[yshift=-18pt] \footnotesize $0$}] at (-1, -1.4) {};
		\node[label={[yshift=-18pt] \footnotesize $e_1$}] at (0, -1.4) {};
		\node[label={[yshift=-18pt] \footnotesize $e_2$}] at (-.5, -2.3) {};
		\node[label={[yshift=-18pt] \footnotesize $0$}] at (.5, -2.3) {};
	\end{tikzpicture}%
\end{center}

Similarly, the diagram shows that $e_2$ is contained in the ray $(10)^\infty$. 
\end{example}

\begin{lemma} \label{lem:vec-rays-ortho}
	Let $\xi, \eta \in \fH$ be two vectors contained in rays $p,q$, respectively. If $p \neq q$ then $\xi, \eta$ are orthogonal vectors.
	
	In particular, if $\xi \in \fH$ is contained in a ray, then $p$ is eventually periodic and the length of a period of $p$ is smaller or equal to $\dim(\fH)$.
\end{lemma}

\begin{proof}
	 We will show the first statement by identifying $\xi, \eta$ with their images inside $\Pi(\fH)$. Since $p \neq q$, there exists $n \in \N$ such that $p_n \neq q_n$. 	 
	 We have:
	\[\langle \xi, \eta \rangle = \langle [e, \xi] , [e, \eta] \rangle = \langle [t_n, \Phi(t_n)\xi], [t_n, \Phi(t_n)\eta] \rangle\]
	where recall $\Phi$ is the functor $\cF \rightarrow \Hilb$ used to construct the P-representation.
	Now, $(t_n,\Phi(t_n)\xi)$ corresponds to a map $\Leaf(t_n)\to \fH$ supported at $p_n$. 
	Since $p_n\neq q_n$ we deduce that $(t_n,\Phi(t_n)\xi)$ and $(t_n,\Phi(t_n)\eta)$ have disjoint support. Therefore, $\langle\xi,\eta\rangle=0$.
	
	Consider now a periodic ray $p$ and $\xi\in\fH$ contained in $p$. If $p$ has a period of length $m$, then note that the ${}_np$ are all distinct for $0\leq n\leq m-1$. Since $p_n\xi$ is contained in ${}_np$ we deduce that the $p_n\xi$ are mutually orthogonal giving $m\leq \dim(\fH)$, and that no vector can be contained in an aperiodic ray.
\end{proof}

We deduce below a strong dichotomy using the compactness of the unit ball of $\fH$.

\begin{proposition} \label{prop:dichotomy convergence prop}
	Let $\fH$ be a P-module. Then either:
	\begin{itemize}
		\item every non-zero vector in $\fH$ is annihilated by all rays; or
		\item there exists a vector contained in some periodic ray $p$.
	\end{itemize}  
\end{proposition}

\begin{proof}
	Consider a P-module $(A,B,\fH)$.
	Assume that there exists a unit vector $\xi$ that is not annihilated by all rays.
	Thus, there exists a ray $p$ satisfying $\lim_{n\to\infty}\|p_n\xi\|\neq 0.$
	Since both $A,B$ have norm smaller than $1$, the sequence $(\|p_n\xi\|)_{n\geq 1}$ is decreasing. 
	It must then converge to a certain $\ell>0$.
The sequence $(p_n\xi)_{n\geq 1}$ is contained in the closed unit ball of $\fH$ which is compact since $\dim(\fH)<\infty$. 
	Therefore, this sequence admits at least one accumulation point $\eta$.
	We necessarily have that $\|\eta\|=\ell$ and thus $\eta$ is non-zero.
	We now construct a new ray $\hat p$ satisfying $\|\hat p_n\eta\|=\ell$ for all $n\geq 1$.
	
	For any $\varep>0$ define $I_\varep$ the set of indices $n\geq 1$ satisfying that $\| \eta - p_n\xi\|<\varep.$
	By definition $I_\varep$ is infinite for any choice of $\varep>0$ (and obviously nested in $\varep$).
	If $n<m$ we write $p_n^m$ for the word in $a,b$ satisfying $p_m = p_n^m\cdot p_n$.
	Observe that if $n,m\in I_\varep$ and $n<m$, then 
	\begin{align*}
		\|\eta-p_n^m\eta\| & \leq \| \eta - p_m\xi\| + \|p_m\xi - p_n^m\eta\| \\
		& \leq \| \eta - p_m\xi\| + \|p_n^m(p_n\xi - \eta)\| \\
		& \leq \| \eta - p_m\xi\| + \|p_n^m\| \cdot \|(p_n\xi - \eta)\| 
	 \leq 2\varep.
	\end{align*}
	This implies that 
	$$\|p_n^m\eta\| \geq \|\eta\| - \|p_n^m\eta -\eta\| \geq \|\eta\| - 2 \varep.$$
	Write $x_k$ the $k$th digits of $p$.
	The previous inequality implies that 
	$$\|x_{n+1} \eta\| \geq \|\eta\| - 2\varep \text{ for all } n\in I_\varep.$$
	From now on we choose $\varep>0$ small enough with respect to $\ell=\|\eta\|$ so that $\|\eta\| - 2\varep> \|\eta\|/\sqrt 2.$
	This implies that $I_\varep\ni n\mapsto x_{n+1}$ is constant.
	Indeed, if both digits $0,1$ would appear as $x_{n+1}$ and $x_{m+1}$ for $n,m\in I_\varep$ we would contradict the Pythagorean equation of $(A,B)$ since we would have:
	$$\|\eta\|^2 = \|A\eta\|^2+\|B\eta\|^2 = \|x_{n+1}\eta\|^2+\|x_{m+1}\eta\|^2> \|\eta\|^2,$$
	a contradiction.
	Let $y_1\in\{0,1\}$ be the digits equal to $x_{n+1}$ for $n\in I_\varep.$
	Observe that $$\|y_1\eta\|\geq \|\eta\|-2\varep' \text{ for all } 0<\varep'\leq\varep.$$
	We deduce that $\|y_1\eta\|=\|\eta\|$.
	A similar reasoning can be applied to the second digits of $p_n^m$ (for $m\geq n+2$ and $n,m\in I_\varep$) and the vector $y_1\eta$ to prove that $n\ni I_\varep\mapsto x_{n+2}$ is constant equal to a certain $y_2$ and moreover $\|y_2y_1\eta\|=\|\eta\|$.
	By induction we obtain a sequence $(y_n)_{n\geq 1}$ and thus a new ray $\hat p$ from it satisfying that 
	\[\|\hat p_k\eta\|=\|y_k\dots y_1 \eta\| = \|\eta\| \text{ for all } k\geq 1. \qedhere\]
\end{proof}

\subsection{Atomic and diffuse P-modules}
The next lemma will permit to only consider periodic rays rather than all {\it eventually} periodic ones.
\begin{lemma} \label{lem:contain-ray-res}
If $\xi\in\fH$ is contained in an eventually periodic ray $p$ that is not periodic, then $\xi \in \fH_{\res}.$
\end{lemma}

\begin{proof}
Consider $p$ and $\xi$ as above.
We have that $p=v \cdot w^\infty$ for some words $w,v$. 
We choose $v$ with the smallest possible length $n$ which is non-zero by assumption.
By restricting to sub-modules we can assume that $\xi$ generates the P-module $\fH$. 
This means that $\fH$ is the span of all the $u\xi$ with $u$ any finite word in binary digits.
Since $\xi$ is contained in $p$ we can restrict to words $u=p_k$ for all $k\geq 0$.
Define $\fK$ as the span of the $p_k\xi$ with $k\geq n$.
Lemma \ref{lem:vec-rays-ortho} shows that $\fK^\perp$ is the span of $p_j\xi$ with $0\leq j< n$.
Therefore, $\fK^\perp$ is residual.
\end{proof}

\begin{definition}
For a P-module $\fH$ define the following subspaces:
\begin{align*}
\fH_{\atom} &:= \textrm{span}\{\xi \in \fH_{\comp} : \xi \textrm{ is contained in some periodic ray } p\} \\
\fH_{\diff} &:= \{\xi \in \fH_{\comp} : \xi \textrm{ is annihilated by all rays}\}.
\end{align*}
Then we say:
\begin{itemize}
\item $\fH_{\atom}$ (resp.~$\fH_{\diff}$) is the atomic (resp.~diffuse) part of $\fH$;
\item $\fH$ is \textit{atomic} (resp.~\textit{diffuse}) if $\fH_{\atom} = \fH_{\comp}$ (resp.~$\fH_{\diff} = \fH_{\comp}$).
\end{itemize}
\end{definition}

\begin{theorem} \label{thm:atomic-diff-decomp}
The subspaces $\fH_{atom}$ and $\fH_{\diff}$ are sub-modules of $\fH$. Furthermore,
\[\fH_{\comp} = \fH_{\atom} \oplus \fH_{\diff}\]
giving the decomposition
\[\fH = \fH_{\atom} \oplus \fH_{\diff} \oplus \fH_{res} \textrm{ and } \sigma = \sigma_{\atom} \oplus \sigma_{\diff}\]
where $\sigma_{\atom} := \sigma_{\fH_{\atom}}$ and $\sigma_{\diff} := \sigma_{\fH_{\diff}}$ are the atomic and diffuse parts of $\sigma$, respectively.
\end{theorem}

\begin{proof}
Let $\xi$ be a non-zero vector in $\fH_{\comp}$. If $\xi$ is contained in a periodic ray $p$, then either $A\xi$ is contained in the periodic ray $_1p$ and $B\xi = 0$ or vice versa. In contrast, if $\xi$ is annihilated by all rays, it is clear that $A\xi$ and $B\xi$ are also annihilated by all rays. Thus, this shows that $\fH_{\atom}$ and $\fH_{\diff}$ are sub-modules. 
	
For the second statement, since $\fH_{\comp}$ decomposes as a direct sum of irreducible sub-modules, it is suffice to show that every irreducible P-module is either atomic or diffuse. Suppose $\fK$ is an irreducible P-module which is not diffuse. Then by Proposition \ref{prop:dichotomy convergence prop} there exists a vector $\eta \in \fK$ which is contained in some periodic ray $p$. Consider the sub-module generated by $\eta$ which must be equal to $\fK$ by assumption of irreducibility of $\fK$. Observe that for a binary word $w$, if $w = p_n$ for some $n$ then $w\eta$ is contained in the ray $_np$, otherwise $w\xi$ is zero. Thus, we deduce $\fK$ is equal to the linear span of the vectors $\{p_k\xi\}_{k \in \N}$ where $p_k\xi$ is contained in the periodic ray $_kp$. Therefore, we can conclude $\fK$ is an atomic P-module. The rest of the theorem follows from Propositions \ref{prop:complete-mod-rep}. 
\end{proof}

\begin{remark}
	Before continuing, it is important to caution the reader of the subtleties present in the above definitions and explain the choice of definitions. 
	\begin{enumerate}[i]
		\item By Lemma \ref{lem:contain-ray-res}, the requirement for the ray to be periodic in the definition of $\fH_{\atom}$ is superfluous. 
		\item If $\fH$ is an atomic (resp.~diffuse) P-module, then this does not necessarily imply that $\fH = \fH_{\atom}$ (resp.~$\fH = \fH_{\diff}$). For example, consider the P-module $(A,B,\C^2)$ with $Ae_1 = Be_2 = e_1$ and $Ae_2 = Be_1 = 0$ (here $e_1,e_2$ denote the standard basis vectors of $\C^2$). Then $\fH_{\comp} = \C e_1$, $\fH_{\res} = \C e_2$, and the vector $e_1$ is contained in the ray $0^\infty$. Thus, $\fH_{\atom} = \C e_1 = \fH_{\comp}$ and $\fH$ is atomic, but the atomic part of $\fH$ does not coincide with $\fH$. An analogous example can be constructed for diffuse P-modules. However, at the level of P-representations, we do define $\sigma$ to be atomic (resp. diffuse) if and only if $\sigma$ is equal to its atomic (resp. diffuse) part. This is because the residual subspace $\fH_{\res}$ is ``forgotten" by the P-representation.
		\item The atomic and diffuse parts of $\fH$ have been defined to be sub-modules of $\fH_{\comp}$. This is to ensure that $\fH_{\atom}$ and $\fH_{\diff}$ are full P-modules and can be decomposed into a direct sum of irreducible atomic and diffuse P-modules, respectively. It is possible to alternatively define the atomic and diffuse parts of $\fH$ as sub-modules $\ti\fH_{\atom}$ and $\ti\fH_{\diff}$, respectively, of $\fH$ by using the same definitions for $\fH_{\atom}$ and $\fH_{\diff}$, respectively, but instead taking $\xi \in \fH$ rather than $\xi \in \fH_{\comp}$. This will still yield a canonical decomposition $\fH = \ti\fH_{\atom} \oplus \ti\fH_{\diff} \oplus \fZ$, where $\fZ$ is some residual subspace, and $\sigma = \ti\sigma_{\atom} \oplus \ti\sigma_{\diff}$. 
		We have that $\ti\sigma_{\atom}\cong \sigma_{\atom}, \ti\sigma_{\diff}\cong \sigma_{\diff},$ and by Lemma \ref{lem:contain-ray-res}, $\fH_{\atom} = \ti\fH_{\atom}$. Additionally, $\fZ\subset \fH_{\res}$ and $\ti\fH_{\diff}\supset \fH_{\diff}$ but these inclusions may be proper.
		\item The notion of diffuse P-modules were first introduced in \cite{Brothier-Wijesena22} where it was also defined for infinite-dimensional P-modules using the same definition. 
		\item The notions of diffuse and atomic are automatically well-defined for \emph{P-modules} but for \emph{P-representations} this requires a non-trivial argument.
		One approach consists in applying Theorem \ref{theo:atom-diff-decomp} and Corollary \ref{cor:intro-ind-weak-mix}.
		Indeed, if a representation is diffuse (resp. atomic) then it is $\NInd$ (not $\NInd$) which is a property preserved by the class of representations.
	\end{enumerate}	
\end{remark}

\begin{center}\textbf{The remainder of the paper will solely study full atomic P-modules and atomic P-representations.}\end{center}
We aim to provide a complete classification up to unitary equivalence for such objects.

\subsection{Atomic sub-representations associated to rays}
We now provide a decomposition of the atomic part of a sub-module using periodic rays.

\begin{definition}
	For a P-module $\fH$ and equivalence class $[p]$ of a periodic ray $p$ define the sub-module
	\[\fH_{\atom}^{[p]} := \textrm{span}\{\xi \in \fH_{\atom} : \xi \textrm{ is contained in } _{k}p \textrm{ for some } k\} \subset \fH\]
	and define the sub-representation $\sigma_{\atom}^{[p]} := \Pi(\fH_{\atom}^{[p]}) \subset \sigma := \Pi(\fH)$. 
\end{definition}

Note that the space $\fH_{\atom}^{[p]}$ is indeed a sub-module since if $\xi$ is contained in the ray $_kp$ then $A\xi$ is either $0$ or contained in $_{k+1}p$, and similarly for $B\xi$. 
Moreover, note that if $p,q$ are periodic and $p\sim q$, then $\fH_{\atom}^{[p]}=\fH_{\atom}^{[q]}.$ Hence, $\fH_{\atom}^{[p]}$ only depends on the class $[p].$

\begin{proposition} \label{prop:atom-decomp-rays}
	Given a P-module $\fH$ with associated P-representation $\sigma$ we have the following direct sum decomposition of the atomic part:
	\[\fHatom = \oplus_{[p]} \fHray{p} \textrm{ and } \sigma_{\atom} = \oplus_{[p]}\sigma_{\atom}^{[p]}\]
	where the direct sum runs over all equivalence classes of periodic rays whose period has length smaller or equal to $\dim(\fHatom)$.
\end{proposition}

\begin{proof}
	Lemma \ref{lem:vec-rays-ortho} shows that if $[p] \neq [q]$ then $\fHray{p}, \fHray{q}$ are orthogonal sub-modules. By definition, every vector in $\fHatom$ belongs in $\fHray{p}$ for some periodic ray $p$. Therefore, we immediately obtain the decomposition:
	\[\fHatom = \oplus_{[p]} \fHray{p} \textrm{ and } \sigma_{\atom} = \oplus_{[p]}\sigma_{\atom}^{[p]}\]
	where the direct sum runs over all equivalence classes of periodic rays.
Lemma \ref{lem:vec-rays-ortho} implies that if $|p|>\dim(\fH_{\atom})$, then $\fH^{[p]}_{\atom}=\{0\}$.	
\end{proof}

\subsection{Irreducible classes of atomic P-modules} \label{subsec:atomic-pmod-model}

We construct explicit representatives of irreducible classes of atomic P-modules.

\begin{notation} \label{not:atom-p-mod-ray}
	Fix $W_d$ to be a set of representatives of all prime binary words $w$ of length $d$ up to cyclic permutation (by prime we mean there does not exist a word $v$ and $n > 1$ such that $w = v^n$). Set $W := \cup_{d \geq 1} W_d$.
	Consider a prime word $w \in W_d$. 
	Write $\{e_1,\dots,e_d\}$ for the standard basis of $\C^d$.
	Define matrices $A_w,B_w \in B(\C^d)$ such that $(A_w)_{n+1,n} = 1$ if the $n$th digit of $w$ is $0$, otherwise $(B_w)_{n+1,n} = 1$ for $n=1,\dots, d$ (by $(A_w)_{d+1,d}$ we mean $(A_w)_{1,d}$). Set all other entries of $A_w$ and $B_w$ to be $0$. Thus, $A_w,B_w$ represent partial shift maps such that $A_w+B_w$ maps $e_j$ to $e_{j+1}$ with $j$ modulo $d$.
	For $\varphi \in S_1$ define $D_\varphi$ to be $d$ by $d$ diagonal matrix whose diagonal entries are all ones except for the last entry which is $\varphi$. We then set
	\[m_{w, \varphi} := (A_wD_\varphi, B_wD_\varphi, \C^d).\] 
\end{notation}

\begin{proposition} \label{prop:atom-p-mod-classif}
Consider two atomic P-modules $m_{w, \varphi}$ and $m_{\ti w, \ti\varphi}$ for any two prime words $w,\ti w \in W$ and unital scalars $\varphi,\ti\varphi\in S_1$. Then we have the following:
\begin{enumerate}[i]
\item $m_{w, \varphi}$ is an irreducible P-module and is equal to $(\C^d)_{\atom}^{[p]}$ where $p = w^\infty$ and $d = \vert w \vert$;
\item If $\fH$ is atomic and irreducible then $\fH\simeq m_{v,\lambda}$ for some $v\in W_d$, $\lambda\in S_1$ with $d = \dim(\fH)$;
\item $m_{w, \varphi}$ and $m_{\ti w, \ti\varphi}$ are equivalent if and only if $(w, \varphi) = (\ti w, \ti\varphi)$. 
\end{enumerate}	
\end{proposition}

\begin{proof}
\textit{Proof of i.} Let $\xi=\sum_{i=1}^d \alpha_i e_i\in \C^d$ be a non-zero vector and let $k$ be the smallest number such that $\alpha_k\neq 0$. Define $\ti w$ to be the cyclic permutation of $w$ so that the first digit of $\ti w$ is the $k$th digit of $w$.  Of the basis elements $\{e_i\}_i$, only $e_k$ is contained in the ray $\ti w^\infty$. Thus, $\ti w\xi = \alpha_k \varphi e_k$. We now apply $(A_w+B_w)D_\varphi$ which is the shift map times $D_\varphi$. We obtain all the basis elements $e_j$ proving that $m_{w,\varphi}$ is irreducible.
	
\textit{Proof of ii.}
Consider an irreducible and atomic P-module $\fH$. By Proposition \ref{prop:atom-decomp-rays} there exists $v\in W_d$ such that $\fH=\fH_{\atom}^{[q]}$ where $d \leq \dim(\fH)$ and $q=v^\infty$.
Since $\fH\neq\{0\}$ there exists $\xi$ in it of norm one. Up to applying a sub-word of $q$ to $\xi$ we may assume that $\xi$ is contained in $q$. Lemma \ref{lem:vec-rays-ortho} implies that $\Xi:=\{\xi, q_1\xi,\cdots, q_{d-1}\xi\}$ are pairwise orthogonal and of norm one. Moreover, $q_d\xi$ is orthogonal to $q_i\xi$ for $i = 1, \dots, d-1$. 
Irreducibility of $\fH$ forces to have $q_d\xi=\lambda\xi$ for some $\lambda\in S_1$. Hence, $\Xi$ is an orthonormal basis of $\fH$. 
By taking matrices over $\Xi$ we obtain $\fH\simeq m_{v,\lambda}.$
	
	\textit{Proof of iii.} Let $U : C^{\vert w \vert} \rightarrow C^{\vert \ti w \vert}$ be a unitary intertwinner between $m_{w, \varphi}$ and $m_{\ti w, \ti\varphi}$. 
	Write $d$ for $|w|$ and note that $d=|\ti w|$ since $U$ is unitary.	
	If $\xi \in \C^{d}$ is contained in a ray $p$ then the intertwinning conditions gives that $U\xi$ must also be contained in $p$. This immediately implies that $\ti w$ is some cyclic permutation of $w$ and thus $w = \ti w$ by definition of $W$. 
	Since $A_w+B_w$ is the shift operator $S$ we deduce that $US D_\varphi = SD_{\ti \varphi}U$.	
	Taking the determinant yields $\varphi=\ti \varphi.$
\end{proof}

\begin{corollary} \label{cor:atom-p-mod-class}
	The set of equivalence classes of irreducible atomic P-modules is in bijection with $W \times S_1$. Geometrically, this is a disjoint union of circles indexed by $W$.
\end{corollary}

\begin{remark}
One could define a more general class of atomic P-modules by
	\[m_{w,D} = (A_wD, B_wD, \C^d)\]
where $D$ is \textit{any} $d$ by $d$ diagonal unitary matrix. However, we do not obtain more irreducible classes of P-modules.
Indeed, if $U$ is the unitary diagonal matrix with $j$th coefficient $\prod_{l=1}^{j-1} D_{l,l}$, then $U$ conjugates $m_{w,D}$ with $m_{w,\varphi}$ where $\varphi=\det(D)$.
\end{remark}

\subsection{Examples of decomposition of Pythagorean representations} \label{subsec:ex-atom-p-rep}

To conclude this section we provide some instructive examples of the decomposition of P-modules.

\begin{example}
	Consider the P-module $m = (A,B,\C^4)$ with $A,B$ given by the below matrices
	\[A = 
	\begin{pmatrix}
		1 & 0 & 0 & 0 \\
		0 & 0 & i & 0 \\
		0 & 0 & 0 & \frac{1}{\sqrt{2}} \\
		0 & 0 & 0 & 0
	\end{pmatrix}
	, \quad B = 
	\begin{pmatrix}
		0 & 0 & 0 & \frac{1}{\sqrt{2}} \\
		0 & 0 & 0 & 0 \\
		0 & -i & 0 & 0 \\
		0 & 0 & 0 & 0
	\end{pmatrix}.
	\]
The vector $e_1$ is contained in the ray $\ell := 0^\infty$ which is the ray going down the left side of $t_\infty$ while $e_2, e_3$ are contained in the zig-zag rays $p:=(10)^\infty$ and ${}_1p=(01)^\infty$, respectively. 
Hence, $\C e_1$ and $\C e_2 \oplus \C e_3$ are atomic sub-modules of $\C^4$. However, since $Ae_4 = 1/\sqrt{2}e_3$ and $Be_4 = 1/\sqrt{2}e_1$, $\C e_4$ is not a sub-module. Therefore, we obtain the following decomposition: 
	\[\C^4 = (\C^4)_{\atom}^{[\ell]} \oplus (\C^4)_{\atom}^{[p]} \oplus (\C^4)_{\res},\]
	where 
	\[(\C^4)_{\atom}^{[\ell]} = \C e_1,\ (\C^4)_{\atom}^{[p]} = \C e_2 \oplus \C e_3,\ (\C^4)_{\res} = \C e_4.\]
	By Proposition \ref{prop:atom-p-mod-classif} we obtain that the complete sub-module of $m$ is equivalent to $m_{0, \varphi} \oplus m_{10, \varphi}$ where $\varphi$ is the scalar $1\in S_1$. By Proposition \ref{prop:complete-mod-rep} we deduce 
	\[\Pi(m) \cong \Pi(m_{0, \varphi}) \oplus \Pi(m_{10, \varphi}).\]	
	The following section will provide a precise classification for the above sub-representations.
\end{example}

\begin{example}
	Consider the P-module $m = (A,B,\C^4)$ where	\[A =
	\begin{pmatrix}
		\frac{1}{\sqrt{2}} & \frac{1}{2} & 0 & 0 \\
		0 & \frac{1}{2} & 0 & 0 \\
		0 & 0 & 0 & 0 \\
		0 & 0 & 1 & 0
	\end{pmatrix}
	,\ B = 
	\begin{pmatrix}
		\frac{1}{\sqrt{2}} & -\frac{1}{2} & 0 & 0 \\
		0 & \frac{1}{2} & 0 & 0 \\
		0 & 0 & 0 & 0 \\
		0 & 0 & 0 & 1
	\end{pmatrix}
	.\]
	We have that
	\[(\C^4)_{\comp} = \C e_1 \oplus \C e_4,\ (\C^4)_{\res} = \C e_2 \oplus \C e_3.\]
	Additionally, $(\C^4)_{\diff} = \C e_1$ and $(\C^4)_{\atom} = \C e_4$ as the vector $e_4$ is contained in the ray $r := 1^\infty$. Note that $\C e_1 \oplus \C e_2$ forms a sub-module of $m$ and is diffuse, but is not equal to \textit{the} diffuse sub-module of $\C^4$ since it contains a residual subspace. Similarly, $\C e_3 \oplus \C e_4$ forms a sub-module of $m$ and is atomic with $e_3$ being contained in the ray $0\cdot r=0\cdot 1^\infty$. However, $\C e_3\subset (\C^4)_{\res}$. 
	The sub-module $\C e_4$ is equivalent to $m_{w, \varphi}$ where $w$ is the prime word $1$ and $\varphi$ the scalar $1$. Hence, we obtain:
	\[\Pi(m) \cong \Pi(m_{w, \varphi}) \oplus \Pi(1/\sqrt{2}, 1/\sqrt{2}, \C).\]
	The diffuse representation $\Pi^F(1/\sqrt{2}, 1/\sqrt{2},\C)$ is the Koopman representation of $F\act [0,1]$ as explained in \cite[Section 6.2]{Brothier-Jones19a}.
\end{example}

%%%%%%%%%%%%%%%%%%%%END STRUCTURE%%%%%%%%%%%%%%%%%%%%%%%%%%%%%%%%%
%%%%%%%%%%%%%%%%%%%%%%%%%%%%%%%%%%%%%%%%%%%%%%%%%%%%%%%%%%%

%%%%%%%%%%%%%%%%%%%%%%%%%%%%%%%%%%%%%%%%%%%%%%%%%%%%%%%%%%%%%%%%%%
%%%%%%%%%%%%%%%%%CLASSIFICATION%%%%%%%%%%%%%%%%%%%%%%%%%%%%%%%%%

\section{Classification of Atomic Pythagorean Representations} \label{rep from rays section}

In this section we describe and classify the atomic part of arbitrary P-representations. For this section $X$ denotes any of the Thompson groups $F,T,V$ unless specified otherwise.

\subsection{Family of one-dimensional representations of the Thompson groups and their parabolic subgroups.} \label{subsec:atomic-rep-fam}
Recall the family of parabolic subgroups $X_p:=\{g\in X:\ g(p)=p\} \subset X$ given by periodic rays $p=w^\infty$ and let $|p|$ be the length $w$. Each of these subgroups are proper except for $F_{0^\infty}, F_{1^\infty}$ which are both equal to $F$. 
Define the family $\{\chi_\varphi^{X,p}\}_{\varphi \in S_1}$ of one-dimensional representations of $X_p$ given by
\[\chi_\varphi^{X,p}(g) = \varphi^{\log(2^{\vert p \vert})(g'(p))} \textrm{ for all $g \in X_p$}\]
where $\log(2^{\vert p \vert})$ is the logarithm function in base $2^{\vert p \vert}$ and recall the definition of the derivative from Section \ref{subsec:F-def}. 
To lighten notation, we shall drop the super-script $X$.

We consider the following class of monomial representations of $X$:
\[\{\Ind_{X_p}^X \chi_\varphi^p: \varphi \in S_1 \textrm{ and } p \textrm{ is an eventually periodic ray}\}.\]
Each of the representations are irreducible by Lemma \ref{lem:monomial-rep-class}. 
When $(X,\vert p \vert) \neq (F,1)$, then the above representations are infinite-dimensional since $X_p\subset X$ has infinite index and otherwise $\Ind_{F_p}^F \chi_\varphi^p = \chi_\varphi^p$ is one-dimensional. 
When $(X,\vert p\vert) \neq (F,1)$, then the equivalence class of $\Ind_{X_p}^X \chi_\varphi^p$ only depends on the ray $p$ up to finite prefixes and $\varphi$ (see Lemma \ref{lem:monomial-rep-class}).

\begin{definition} \label{monomial rep in Pythag rep def}
Write $\ell:=0^\infty,r:=1^\infty$ for the endpoints of $\cC$ and let $d\geq 1$ be a natural number.
Define:
\begin{align*}
R^F_{1} &= \{ \chi_\varphi^\ell \oplus \Ind_{F_{1\cdot \ell }}^F\chi_\varphi^{1\cdot\ell}:\ \varphi \in S_1\} 
\cup \{ \chi_\varphi^r \oplus \Ind_{F_{0\cdot r }}^F\chi_\varphi^{0\cdot r}:\ \varphi \in S_1, \varphi \neq 1 \},\\
R^X_{d} &= \{ \Ind_{X_p}^X\chi_\varphi^p:\ p = w^d \textrm{ for } w \in W_d,\ \varphi \in S_1\},\ (X,d) \neq (F,1).
\end{align*}
\end{definition}

\subsection{Decomposition of atomic P-representations} \label{subsec:atom-rep-decomp}
Propositions \ref{prop:atom-decomp-rays} and \ref{prop:atom-p-mod-classif} showed that all complete atomic P-modules can be decomposed into a finite direct sum of irreducible atomic P-modules of the form $m_{w,\varphi}$ for some $w \in W, \varphi \in S_1$. Hence, to classify all atomic representations, it is sufficient to only classify the class of representations $\{\sigma^Y_{w, \varphi}\}_{w \in W, \varphi \in S_1}$ where $\sigma^Y_{w, \varphi} := \Pi^Y(m_{w, \varphi})$ and $Y = F,T,V,\cO$.

\begin{theorem} \label{thm:atom-rep-class}
	Take $X = F,T,V$, $\varphi \in S_1$, $w \in W$ and set $p := w^\infty$. We have the following.
	\begin{enumerate}[i]
		\item If $(X,\vert w \vert) \neq (F,1)$ then 
		\[\sigma^X_{w, \varphi} \cong \Ind_{X_p}^X \chi_{\varphi}^p.\]
		\item If $(X, \vert w \vert) = (F,1)$ then 
		\[\sigma^F_{w, \varphi} \cong \chi_\varphi^p \oplus \Ind_{F_{q}}^F \chi_{\varphi}^q\]
		where $q=1\cdot \ell$ if $w=0$ and $q=0\cdot r$ when $w=1$.
		\item A representation of $X$ is atomic if and only if it is a finite direct sum of ones belonging in $R^X_{d}$ for $d \geq 1$.
		\item If $\Pi^X(\fH) \cong \oplus_{d,j} \pi_{d,j}$ with $\pi_{d,j} \in R^X_d$ and $j$ is in some index set $J_d$, then $\sum_{d,j} d = \dim(\fH_{\atom})$. 
	\end{enumerate}
\end{theorem}

\begin{proof}
	Most of the proof consists in showing the first two statements. This will be achieved by finding suitable cyclic vectors and by comparing matrix coefficients.
	Fix $d\geq 1$ and consider an atomic P-module $m_{w, \varphi}=(A_wD_\varphi,B_wD_\varphi,\C^d)$ for $(w,\varphi)\in W_d\times S_1$ and take $X$ in $F,T,V$. Hence, $|w|=d$.
	Recall that $(e_1,\dots,e_d)$ is the standard basis of $\C^d$.

	{\bf We first assume that $(X,d)\neq (F,1)$.}
	Define $p := w^\infty$ and $\sigma := \sigma_{w, \varphi}^X$.
	Here we identify $w$ with a vertex in $\Ver$.
	Observe that $w$ lies in the ray $p$ and by definition of $m_{w, \varphi}$ we have
	$$\tau_{w^n}(e_1) = \varphi^n e_1 \textrm{ and } \tau^*_{w^n}(e_1) = \varphi^{-n}e_1 \text{ for all } n \geq 1.$$
		
	\textbf{Claim 1.} The vector $e_1$ is cyclic for $\sigma$. 
	
	It is suffice to show that $\tau_v^*(e_j) \in \Span\{\sigma(X)e_1\}$ for all words $v$ and all $1\leq j\leq d$. 
	Fix such $v$ and $j$. 
	Note that $e_1=\varphi\tau_w^*(e_1)$ and moreover there exists a sub-word $x$ of $w$ and a scalar $\lambda\in S_1$ such that $e_1=\tau_x(e_j).$
	Moreover, if $g=[t,\kappa,s]\in T$ and $y,w$ are corresponding leaves, then 
	$$\sigma(g)(e_1) = \tau_{y}^* \tau_w \varphi \tau_w^*(e_1) = \varphi \tau_y^*(e_1)=\varphi\lambda \tau_y^*\tau_x(e_j).$$
	Since $T$ acts transitively on the non-trivial sdi's and $w$ is non-trivial we can choose $y$ to be anything we want.
	Taking $y:=vx$ yields the result.
	This proves the $T$-case. The $V$-case follows since $T\subset V$.
	Now, for the $F$-case we proceed in the same way. Since $|w|\neq 1$ and is prime we must have that $w$ does not lie on any of the two endpoints of $\cC$.
	We then use the fact that $F$ acts transitively on sdi's that do not contain an endpoint of $\cC$.

	Denote $\theta_p$ to be the cyclic representation of $X_p$ given by the sub-representation of $\sigma\restriction_{X_p}$ (the restriction of $\sigma$ to the subgroup $X_p$) generated by $e_1$.

	\textbf{Claim 2:} The representation $\theta_p$ of $X_p$ is equivalent to $\chi_{\varphi}^p$.
	
	We shall make use of the description of $X_p$ using tree-diagrams as discussed in Section \ref{parabolic desc subsection}.
Let $g := [t,\kappa,s] \in X_p$. Since $g(p)=p$ there exists $i,j\geq 1$ such that $g$ restricts into $p_i\cdot z\mapsto p_j\cdot z.$ Moreover, we must have ${}_ip={}_jp$ implying that $|i-j|= |w|n$ for some $n\in\Z.$
We deduce that $\sigma(g)e_1=\varphi^n e_1.$
Hence, $\theta_p$ is one-dimensional and equivalent to $\chi_\varphi^p$.
	
	Now consider $g := [t,\kappa,s] \in X$ and set $\Leaf(t) = \{\nu_i\}_{i \in I}$ and $\Leaf(s) = \{\omega_i\}_{i \in I}$. Let $\nu_k, \omega_l$ be the leaves of $t$ and $s$, respectively, which the ray $p$ passes through. Denote 
	$$\phi : X\to \R, \  g \mapsto \langle \sigma(g)e_1, e_1\rangle$$ the associated matrix coefficient.
	
	\textbf{Claim 3:} If $g \notin X_p$, then $\phi(g) = 0$.
	
	By the discussion in Section \ref{parabolic desc subsection}, $g \notin X_p$ if and only if $k \neq l$ (equivalently, $\nu_k$ and $\omega_l$ are not corresponding leaves of $g$) or $k = l$ and Equation \ref{parabolic subgroup condition eqn} is not satisfied. First suppose $k \neq l$. 	
	Then by Equation \ref{action rearrange equation}:
	\begin{align*}
		\phi(g) &= \langle \sigma(g)e_1, e_1 \rangle 
		= \sum_{i \in I} \langle \tau_{\omega_i}(e_1), \tau_{\nu_i}(e_1) \rangle \\
		&= \sum_{i \neq k,l} \langle \tau_{\omega_i}(e_1), \tau_{\nu_i}(e_1) \rangle + \langle \tau_{\omega_k}(e_1), \tau_{\nu_k}(e_1) \rangle + \langle \tau_{\omega_l}(e_1), \tau_{\nu_l}(e_1) \rangle.  
	\end{align*}
	From Observation \ref{contained vector norm obs}, $\tau_{\nu_i}(e_1) = 0$ for $i \neq k$ since $\nu_i \notin \Ver_{p}$. Similarly $\tau_{\omega_j}(e_1) = 0$ for $j \neq l$. Thus, each of the terms in the above equation is $0$ and we obtain $\phi(g) = 0$. 
	
	Then suppose $k = l$ and Equation \ref{parabolic subgroup condition eqn} is not satisfied for $m = \length(\nu_k)$ and $n = \length(\omega_l)$. Subsequently, $\nu_k, \omega_l$ are corresponding leaves and $m-n \notin d\N$. Then by a similar reasoning as before we have
	\begin{equation*}
		\phi(g) = \sum_{i \neq k} \langle \tau_{\omega_i}(e_1), \tau_{\nu_i}(e_1) \rangle + \langle \tau_{\nu_k}(e_1), \tau_{\omega_k}(e_1) \rangle = \langle p_ne_1, p_me_1 \rangle.
	\end{equation*}
	Since $m-n \notin d\N$ and the length of a period of $p$ is $d$ it follows that $p_ne_1, p_me_1$ are vectors contained in different rays. Therefore $\langle p_ne_1, p_me_1 \rangle = 0$ by Lemma \ref{lem:vec-rays-ortho} and thus $\phi(g) = 0.$ 

	The three claims yield item (i).
	
	{\bf We shall now treat the remaining case $(X,d) = (F,1)$.}
	There are only two cases to consider here: $w = 0$ or $w = 1$. 
	We consider the first. The second one follows via a similar proof.	
	Set $p = \ell=0^\infty$ and $q = 1\cdot \ell=1\cdot 0^\infty$. Note now that $m_{w,\varphi}=(\varphi,0,\C)$.
	Write $\xi$ for a unit vector of $\C$. The dense subspace $\scrK$ of $\scrH$ are then trees with leaves decorated by scalars. 
	Given $g=[t,s]\in F$ we have that 
	$$\sigma(g) \xi=(\tau_0^k)^* \tau_0^n \xi=\varphi^{n-k}\xi$$ where $k=|y|,n=|x|$ and $y,x$ are the first leaves of $t,s$, respectively. 
	We deduce that $\scrH_1:=\C \xi\subset \scrH$ defines a sub-representation $\sigma_1$ of $\sigma$ equivalent to $\chi_\varphi^p$.\\
	Consider now the vector $\tau_1^*(\xi)$. It is orthogonal to $\xi$ inside $\scrH$ and thus span a sub-representation $(\sigma_2,\scrH_2)\subset (\sigma,\scrH)$  so that $\scrH_1\perp \scrH_2.$
	Consider any word $v\neq 0^n$ for all $n$.
	By transitivity of the action of $F$ on sdi's not containing endpoints there exists $g=[t,s]\in F$ where $v\cdot 0$ and $10$ are corresponding leaves of $(t,s)$.
	Then, $\sigma(g) \tau_1^*(\xi) = \tau_v^*(\xi)$.
	The space $\scrH$ is the closed linear span of the $\tau_u^*(\xi)$ with $u$ any word.
	Moreover, note that $\tau_0^*(\xi)=\varphi^{-1} \xi$.
	From there we deduce that $\scrH=\scrH_1\oplus\scrH_2.$
	By slightly adapting the proof of item (i) we deduce that $\sigma_2\simeq \Ind_{F_q}^F\chi_\varphi^q$. 
	This yields item (ii).

	The third and fourth statements follow from the fact that: P-functors preserve direct sums, the classification of all atomic irreducible P-modules from Proposition \ref{prop:atom-p-mod-classif}, and the first two statements of the theorem.
\end{proof}

We can now prove several useful corollaries. 

\begin{proof}[Proof of Corollary \ref{cor:intro-classif-atom-general}]
	Take $X = F,T,V$, $\varphi_i \in S_1$, $w_i \in W$ and set $p_i := w_i^\infty$ for $i = 1,2$. It is elementary to show that all the parabolic subgroups $X_p \subset X$ are self-commensurated. Hence, by the Mackey--Shoda criterion we have that $\Ind_{X_p}^X \chi_{\varphi}^{p}$ is irreducible. Moreover, the representations $\Ind_{X_p}^X \chi_{\varphi_1}^{p_1}, \Ind_{X_p}^X \chi_{\varphi_2}^{p_2}$ are equivalent if and only if either $(w_1, \varphi_1) = (w_2, \varphi_2)$ or $X = F$ and $(\vert w_i \vert, \varphi_i) = (1,1)$ for $i = 1,2$ (note we only require $w_2$ to be a cyclic permutation of $w_1$; however, by definition of $W$ this implies $w_1 = w_2$). Then the corollary for $F,T,V$, and thus $\cO$, immediately follows from Proposition \ref{prop:complete-mod-rep} and Theorem \ref{thm:atom-rep-class}.
\end{proof}

Recall that a representation of a group $G$ is weakly mixing (resp.~Ind-mixing) when it does not contain any (resp.~induction of a) non-zero finite-dimensional representation.

\begin{proof}[Proof of Corollary \ref{cor:intro-ind-weak-mix}]
	Consider a P-module $m=(A,B,\fH)$ and $X=F,T,V$.
	Using the main results of our previous article we only need to prove the two reverse implications that is: $\Pi^X(m)$ weak-mixing (resp.~$\NInd$) implies either $X=T,V$ or $\lim_nA^n\xi=\lim_nB^n\xi=0$ (resp.~$\lim_n p_n\xi=0$ for all rays $p$) for all vectors $\xi\in\fH$ \cite{Brothier-Wijesena22}.
	
	Assume $X=F$ and there exists $\xi\in\fH$ so that $\lim_nA^n\xi\neq 0$. 
	Then the proof of Proposition \ref{prop:dichotomy convergence prop} implies there exists a vector $\eta \in \fH$ contained in the ray $\ell=0^\infty$. Then from Theorem \ref{thm:atom-rep-class} we obtain that $\Pi^F(m)$ contains a one-dimensional representation $\chi_{\varphi}^\ell$. Hence, $\Pi^F(m)$ is not weakly mixing. 
	A similar proof works by swapping $A$ by $B$ and the endpoints of $\cC$. 
	For $X = T,V$ Theorem \ref{thm:atom-rep-class} shows that $\Pi^X(m)$ is always weak-mixing.
	This proves the first statement of the corollary.
	
	Assume now that there exists $\xi\in\fH$ and a ray $p$ so that $\lim_n p_n\xi \neq 0$. Applying Proposition \ref{prop:dichotomy convergence prop} again implies there exists a vector $\eta \in \fH$ contained in a periodic ray $q$. By Theorem \ref{thm:atom-rep-class}, $\Pi^X(m)$ contains a monomial representation induced from a parabolic subgroup of $X$. Therefore, $\Pi^X(m)$ is not $\NInd$. 
\end{proof}

\subsection{Manifolds of atomic representations}\label{sec:geometry}
From the above proof of Corollary \ref{cor:intro-classif-atom-general} we can deduce that for an atomic P-module $\fH$, if $\fH_{\comp} \cong m_{w, \varphi}$ for some $\varphi \in S_1, w \in W$ then $\vert w \vert$ is an invariant for $\Pi(\fH)$. That is, if $v$ is any prime word (not necessarily in $W$) and $\lambda \in S_1$ such that $\fH_{\comp} \cong m_{v, \lambda}$ then necessarily $\vert v \vert = \vert w \vert$ (i.e.~they have minimal periods of same length). 
Hence, this provides a dimension number $\dim_{\cP}(\Pi(\fH)) = \dim_{\cP}(\fH) := \vert w \vert$ for both the atomic representation and underlying P-module which we term as the Pythagorean dimension (P-dimension for short). Observe the P-dimension coincides with the usual dimension of $\fH_{\comp}$ as a complex vector space. Using the results from the previous subsection, we obtain a powerful classification result for atomic representations for each P-dimension.

Fix $d\geq 1$ and consider the Hilbert space $\C^d$ equipped with its standard basis.
Let $PM(d)$ be the set of all P-modules $(A,B,\C^d)$ where now $A,B$ are $d$ by $d$ matrices.
The group $\PSU(d)$ acts by conjugation on $PM(d)$ and by definition two P-modules are equivalent if they are in the same $\PSU(d)$-orbit.
Define now $\Irr_{\atom}(d)\subset PM(d)$ as the subset of irreducible atomic P-modules.
It is of course globally stabilised by $\PSU(d)$.
Section \ref{subsec:atomic-pmod-model} implies that $\{m_{w,\varphi}:\ (w,\varphi)\in W_d\times S_1\}$ forms a set of representatives of the orbit space of $\Irr_{\atom}(d)$.
Then Corollary \ref{cor:intro-classif-atom-general} shows that if $(X,d)\neq (F,1)$, then $\Pi^X(m_{w,\varphi})$ is irreducible for all $(w,\varphi)\in W_d\times S_1$ and moreover $\Pi^X$ preserves equivalence classes.
All together this proves Corollary \ref{lettercor:geometry}.

%%%%%%%%%%%%%%%%%%%%%%%%%%%%%%%%%%%%%%%%%%%%%%%
%%%%%%%%%%%%%%%BIBLIOGRAPHY%%%%%%%%%%%%%%%%%%%%%%

\newcommand{\etalchar}[1]{$^{#1}$}

\end{document}